\documentclass[a4paper,reqno]{amsart}

\usepackage[american]{babel}
\usepackage[T1]{fontenc}
\usepackage[utf8]{inputenc}
\usepackage{amsmath,amssymb,amsfonts,amsthm,amstext}
\usepackage{dsfont,mathtools,paralist,mathrsfs}
\usepackage{microtype}
\usepackage[numbers,sort&compress]{natbib}
\usepackage{hyperref}
\usepackage{ifthen}
\usepackage{esint}

\newcommand{\dx}{\mathrm{d}}
\newcommand{\apply}[3][]{\left<#2,#3\right>\ifthenelse{\equal{#1}{}}{}{_{#1}}}
\newcommand{\scalar}[3][]{\left(#2\mid#3\right)\ifthenelse{\equal{#1}{}}{}{_{#1}}}
\renewcommand{\phi}{\varphi}
\newcommand{\e}{\mathrm{e}}
\newcommand{\setone}{\mathds{1}}
\newcommand{\eps}{\varepsilon}
\renewcommand{\rho}{\varrho}

\DeclareMathOperator*{\esssup}{ess\,sup}

\DeclareMathOperator{\hull}{span}
\DeclareMathOperator{\AP}{AP}
\DeclareMathOperator{\AAP}{AAP}
\DeclareMathOperator{\BUC}{BUC}
\DeclareMathOperator{\Freq}{Freq}

\theoremstyle{definition}
\newtheorem{theorem}{Theorem}[section]
\newtheorem{proposition}[theorem]{Proposition}
\newtheorem{corollary}[theorem]{Corollary}
\newtheorem{lemma}[theorem]{Lemma}
\newtheorem{remark}[theorem]{Remark}
\newtheorem{example}[theorem]{Example}
\newtheorem{definition}[theorem]{Definition}

\title{Inhomogeneous Parabolic Neumann Problems}
\date{\today}
\author{Robin Nittka}
\address{Robin Nittka\\University of Ulm\\Institute of Applied Analysis\\89069 Ulm\\Germany}
\email{robin.nittka@uni-ulm.de}

\numberwithin{equation}{section}

\begin{document}
\begin{abstract}
	We study second order parabolic equations on Lipschitz domains subject to inhomogeneous Neumann (or, more generally, Robin) boundary
	conditions. We prove existence and uniqueness of weak solutions and their continuity up to the boundary of the parabolic
	cylinder. Under natural assumptions on the coefficients and the inhomogeneity we can also prove
	convergence to an equilibrium or asymptotic almost periodicity.
\end{abstract}

\maketitle

\section{Introduction}

Let $\Omega$ be a bounded Lipschitz domain in $\mathds{R}^N$.
Our model problem is the
the heat equation
\[
	\left\{ \begin{aligned}
		u_t(t,x) - \Delta u(t,x) & = f(t,x), && t > 0, \, x \in \Omega \\
		\frac{\partial u(t,z)}{\partial \nu} & = g(t,z), && t > 0, \, z \in \partial\Omega \\
		u(0,x) & = u_0(x), && x \in \Omega
	\end{aligned} \right.
\]
subject to inhomogeneous Neumann boundary conditions.
The above problem has a unique weak solution in an $L^2$-sense if $f$, $g$ and $u_0$ are square-integrable.
We are interested in its regularity at the boundary and its asymptotic behavior.
Such problems appear in a natural way
in control theory~\cite{BDKL01,CGS93} or thermal imaging~\cite{CKY99}.

More precisely, we show the following: if $u_0$ is continuous
and $f$ and $g$ satisfy some integrability conditions,
then the solution $u$ is continuous up to the boundary of the parabolic cylinder;
if $f$ and $g$ converge to zero
in a time-averaged sense, then $u$ converges to zero uniformly on $\overline{\Omega}$;
finally, if $f$ and $g$ are almost periodic functions, then $u$
is asymptotically almost periodic with essentially the same frequencies.

Even though the heat equation will be our model case,
we will admit general strongly elliptic operators subject to Robin boundary conditions in all of our results.
For homogeneous boundary conditions, i.e., if $g=0$, these problem are well understood and can
be studied by semigroup methods. Inhomogeneous boundary conditions, however, are more delicate.
For smooth data, some existence and regularity results can be found in~\cite[Theorem~5.18]{Lieberman}
or~\cite{DHP07}. Existence of weak solution is shown in~\cite[\S 4.15.3]{LM72}.
Regularity theory in $L^p$-spaces for the inhomogeneous elliptic Neumann problem can be found for example in~\cite{Kenig1,Kenig2}
and for the parabolic Neumann problem in~\cite{Pruess}, both with a different emphasis.
Asymptotic almost periodicity has been studied in~\cite{Are00} for the inhomogeneous Dirichlet problem.

In order to study the asymptotic behavior we want to follow a semigroup approach
by considering the equation as an abstract Cauchy problem in a suitable space, which is adapted to the
boundary data. To this end one could use spaces of distributions that
contain functionals arising from boundary integrals, a
strategy which has been pursued with negative exponent Sobolev spaces~\cite{HDR09} and Sobolev-Morrey spaces~\cite{Gri07}.
This approach, however, has the disadvantage that a priori the solutions no more regular than generic elements of these spaces,
whereas it would be favorable to have continuous functions as solutions. The parabolic structure of the equation
does not immediately help because a gain in regularity is not obvious in presence of the inhomogeneities.
The regularity matters in particular in the limits $t \to 0$ and $t \to \infty$ since
semigroup methods provide us typically with convergence in the norm of the underlying space.

In view of these considerations we aim towards results in the space $\mathrm{C}(\overline{\Omega})$.
Existence is however much more convenient in $L^2(\Omega)$, which is why we will start out by
considering $L^2$-solution. By using $\mathrm{C}(\overline{\Omega})$ we are able to obtain
uniform convergence of $u$ on $\overline{\Omega}$ as $t \to 0$ and as $t \to \infty$, or more generally asymptotic
almost periodicity. This seems to be completely new for Neumann boundary conditions and is our main result.

\smallskip

Our strategy is the following. When formulating the initial-boundary value problem as an abstract Cauchy problem on $L^2(\Omega)$
or $\mathrm{C}(\overline{\Omega})$, we switch to a product space. More precisely,
we regard the inhomogeneous heat equation as an inhomogeneous abstract Cauchy problem for the operator $A$ given
by $A(u,0) = (\Delta u, -\frac{\partial u}{\partial \nu})$ in the space $L^2(\Omega) \times L^2(\partial\Omega)$.
This operator $A$ is not densely defined and hence not the generator of a strongly continuous semigroup.
In fact, it turns out that $A$ does not even satisfy the Hille-Yosida estimates.
Still, the operator is resolvent positive and hence generates a once integrated semigroup.
This implies existence and uniqueness of solutions for regular right hand sides $f$ and $g$
and gives information about the asymptotic behavior of solutions.
These results can be extended to a larger class of less regular right hand sides
once we obtain suitable a priori estimates.

The idea to consider a non-densely defined operator $A$ on a product space in order to treat
inhomogeneous boundary conditions has first been used
by Arendt for the study of the heat equation with inhomogeneous Dirichlet boundary conditions~\cite{Are00}.
Here we copy the skeleton of his proofs.
The details are however quite different, the main aspects being the following:
\begin{enumerate}[(1)]
\item
	We restrict ourselves to Lipschitz domains, which is the usual framework for Neumann problems,
	whereas one of Arendt's main points are the optimal boundary regularity assumptions.
\item
	Our a priori estimate needs more sophisticated methods, whereas
	for the Dirichlet problem it is a consequence of the parabolic maximum principle.
\item
	The Neumann problem has a smoothing effect with respect to the boundary conditions, which
	allows us to obtain continuous solutions even for non-smooth functions $g$, whereas for Dirichlet
	problems the boundary has to be continuous. The latter fact is reflected in various places. It explains
	for example why for the Neumann problem the solution is asymptotically almost periodic in the sense of Bohr even if the
	right hand side is almost periodic only in the sense of Stepanoff, whereas for the Dirichlet problem
	this cannot hold.
\end{enumerate}

\smallskip

The article is organized as follows. In Section~\ref{sec:sol} we introduce the initial-boundary value problem.
We show existence and uniqueness of solutions and discuss the relationship between three different notions of solutions.
Section~\ref{sec:reg} contains results and pointwise estimates for the solutions as well as their
continuity. The most technical part of this section is however postponed to Appendix~\ref{app:pointest}
in the hope that this improves the readability of the article as a whole.
In Section~\ref{sec:conv} we study the convergence of solutions.
More precisely, we give natural sufficient conditions for the solution to be bounded or to converge to a constant
function.
Finally, in Section~\ref{sec:asymp} we show that for asymptotically almost period right hand sides in the sense
of Stepanoff, the solution is asymptotically almost periodic in the sense of Bohr.

\section{Solutions}\label{sec:sol}

Let $\Omega \subset \mathds{R}^N$ be a bounded Lipschitz domain, $N \ge 2$.
For convenience we assume throughout that $\Omega$ is connected; otherwise
we could consider each connected component separately.
Let $a_{ij} \in L^\infty(\Omega)$, $b_j, c_i \in L^q(\Omega)$,
$d \in L^{\frac{q}{2}}(\Omega)$ and $\beta \in L^{q-1}(\partial\Omega)$ be given, where $q > N$ is arbitrary,
and assume that there exists $\mu > 0$ such that
\begin{equation}\label{eq:elliptic}
	\sum_{i,j=1}^N a_{ij} \xi_i \xi_j \ge \mu |\xi|^2 \quad \text{for all } \xi \in \mathds{R}^N.
\end{equation}
Throughout the article we will always refer to the inhomogeneous Robin problem
\begin{equation}%\label{eq:parabolic}
	(P_{u_0,f,g}) \quad \left\{ \begin{aligned}
		u_t(t,x) - Au(t,x) & = f(t,x), && t > 0, \; x \in \Omega \\
		\frac{\partial u(t,z)}{\partial \nu_A} + \beta u(t,z) & = g(t,z), && t > 0, \; z \in \partial\Omega \\
		u(0,x) & = u_0(x), && x \in \Omega,
	\end{aligned} \right.
\end{equation}
with given $u_0 \in L^2(\Omega)$, $f \in L^2(0,T;L^2(\Omega))$ and $g \in L^2(0,T;L^2(\partial\Omega))$.
Here, at least on a formal level,
\begin{align*}
	Au & \coloneqq \sum_{j=1}^N D_j \Bigl( \sum_{i=1}^N a_{ij} D_iu + b_j u \bigr) - \Bigl( \sum_{i=1}^N c_i \, D_iu + du \Bigr) \\
	\frac{\partial u}{\partial \nu_A} & \coloneqq \sum_{j=1}^N \Bigl( \sum_{i=1}^N a_{ij} D_iu + b_ju \Bigr) \nu_j,
\end{align*}
where $\nu = (\nu_j)_{j=1}^N$ denotes the outer unit normal of $\Omega$ at the boundary $\partial\Omega$.
It is convenient to introduce also the bilinear forms
\begin{align}
	a_0(u,v) & \coloneqq \int_\Omega \sum_{j=1}^N \Bigl( \sum_{i=1}^N a_{ij} D_iu + b_j u \Bigr) D_j v + \int_\Omega \Bigl( \sum_{i=1}^N c_i \, D_iu + du \Bigr) v
	\label{eq:defa0}
\shortintertext{and}
	a_\beta(u,v) & \coloneqq a_0(u,v) + \int_{\partial\Omega} \beta u v
	\label{eq:defab}
\end{align}
for $u$ and $v$ in $H^1(\Omega)$, where $H^1(\Omega)$ refers to the Sobolev space of all functions in $L^2(\Omega)$
whose first derivative also lie in $L^2(\Omega)$.

We introduce and compare various notions for a solution of $(P_{u_0,f,g})$,
which are based on the observation that on a formal level the divergence theorem gives
\begin{equation}\label{eq:form0}
	a_0(u,v) = \int_{\partial\Omega} \frac{\partial u}{\partial \nu_A} \, v - \int_\Omega Au \; v
\end{equation}
for all $v \in H^1(\Omega)$.
A \emph{weak solution} is now defined by testing against a smooth function and formally integrating by parts.
\begin{definition}\label{def:weaksol}
	We say that a function $u \in \mathrm{C}([0,T];L^2(\Omega)) \cap L^2(0,T;H^1(\Omega))$
	is a \emph{weak solution of $(P_{u_0,f,g})$ on $[0,T]$} for some $T > 0$ if
	\begin{equation}\label{eq:weak}
		\begin{aligned}
			& -\int_0^T \int_\Omega u(s) \; \psi_t(s) + \int_0^T a_\beta(u(s), \psi(s)) \\
				& \qquad\qquad = \int_\Omega u_0 \; \psi(0) + \int_0^T \int_\Omega f(s) \, \psi(s) + \int_0^T \int_{\partial\Omega} g(s) \, \psi(s)
		\end{aligned}
	\end{equation}
	for all $\psi \in H^1(0,T; H^1(\Omega))$ that satisfy $\psi(T) = 0$.

	We say that a function $u\colon [0,\infty) \to L^2(\Omega)$ is a \emph{weak solution of $(P_{u_0,f,g})$ on $[0,\infty)$}
	if for every $T > 0$ its restriction to $[0,T]$ is a weak solution on $[0,T]$.
\end{definition}

In order to give two further definitions of a solution, we first introduce the $L^2$-realization $A_2$ of $A$
with Robin boundary conditions, which is also based on~\eqref{eq:form0}.
\begin{definition}\label{def:weakderiv}\mbox{}
\begin{enumerate}[(a)]
\item
	Let $u \in H^1(\Omega)$. We say that $Au \in L^2(\Omega)$ if there exists a (necessarily unique)
	function $f \in L^2(\Omega)$ satisfying $a_0(u,\eta) = -\int_\Omega f \, \eta$
	for all $\eta \in H^1_0(\Omega)$. In this case we define $Au \coloneqq f$.
\item
	Let $u \in H^1(\Omega)$ satisfy $Au \in L^2(\Omega)$.
	We say that $\frac{\partial u}{\partial \nu_A} \in L^2(\Omega)$ if there exists a (necessarily unique) function $g \in L^2(\partial\Omega)$
	satisfying $a_0(u,\eta) = \int_{\partial\Omega} g \, \eta - \int_\Omega Au \; \eta$ for all $\eta \in H^1(\Omega)$.
	In this case we define $\frac{\partial u}{\partial \nu_A} \coloneqq g$.
\item
	We define the operator $A_2$ on the space $L^2(\Omega) \times L^2(\partial\Omega)$ by
	\begin{align*}
		D(A_2) & \coloneqq \Bigl\{ (u,0) : u \in H^1(\Omega), \; Au \in L^2(\Omega), \; \frac{\partial u}{\partial \nu_A} \in L^2(\partial\Omega) \Bigr\} \\
		A_2(u,0) & \coloneqq \Bigl(Au, \; -\frac{\partial u}{\partial \nu_A} - \beta u|_{\partial\Omega} \Bigr).
	\end{align*}
\end{enumerate}
\end{definition}

\begin{remark}\label{rem:formform}
	It is easily checked that $(u,0) \in D(A_2)$ with $-A_2(u,0) = (f,g)$ if and only if
	\[
		a_\beta(u,v) = \int_\Omega f v + \int_{\partial\Omega} g v
	\]
	for all $v \in H^1(\Omega)$.
\end{remark}

It is an exercise in applying H\"older's
inequality, the Sobolev embedding theorems and Young's inequality
to prove that there exists $\omega \ge 0$ such that
\begin{equation}\label{eq:aelliptic}
	a_\beta(u,u) \ge \frac{\mu}{2} \int_\Omega |\nabla u|^2 - \omega \int_\Omega |u|^2
\end{equation}
for all $u \in H^1(\Omega)$. We leave the verification to the reader.

Next we collect a few facts about $A_2$.
\begin{lemma}\label{lem:DA2}
	The operator $A_2$ is resolvent positive.
	More precisely, the operator $\lambda - A_2\colon D(A_2) \to L^2(\Omega) \times L^2(\partial\Omega)$
	is invertible for all $\lambda > \omega$, where $\omega$ is as in~\eqref{eq:aelliptic},
	and if $A_2(u,0) = (f,g)$ with non-negative functions $f \in L^2(\Omega)$
	and $g \in L^2(\partial\Omega)$, then $u \ge 0$ almost everywhere.
	Moreover, if $D(A_2)$ is equipped with the graph norm, then $D(A_2)$ is continuously embedded into $H^1(\Omega) \times \{ 0 \}$.
\end{lemma}
\begin{proof}
	Let $\omega$ be as in~\eqref{eq:aelliptic} and fix $\lambda > \omega$. Then
	\begin{equation}\label{eq:formelliptic}
		\lambda \int_\Omega |u|^2 + a_\beta(u,u)
			\ge \alpha \, \|u\|_{H^1(\Omega)}^2
	\end{equation}
	for all $u \in H^1(\Omega)$ with $\alpha \coloneqq \min\bigl\{ \lambda - \omega, \frac{\mu}{2} \bigr\} > 0$.
	Hence by the Lax-Milgram theorem~\cite[\S 5.8]{GT01} for every $f \in L^2(\Omega)$ and $g \in L^2(\partial\Omega)$
	there exists a unique function $u \in H^1(\Omega)$ such that
	\begin{equation}\label{eq:laxmilgram}
		\lambda \int_\Omega u v + a_\beta(u,v) = \int_\Omega f v + \int_{\partial\Omega} g v
	\end{equation}
	for all $v \in H^1(\Omega)$.
	By Remark~\ref{rem:formform} this means precisely that there is a unique function $u \in H^1(\Omega)$
	with $(u,0) \in D(A_2)$ and
	\[
		(\lambda - A_2)(u,0) = (\lambda u, 0) - A_2(u,0) = (f,g).
	\]

	We have seen that $\lambda - A_2\colon D(A_2) \to L^2(\Omega) \times L^2(\partial\Omega)$ is a bijection for $\lambda > \omega$.
	Assume now that $f \le 0$ and $g \le 0$. Let $(u,0) \coloneqq (\lambda - A_2)^{-1}(f,g)$ and set
	$v \coloneqq u^+ = u \; \setone_{\lbrace u > 0 \rbrace}$. Then
	\[
		D_jv = D_ju \; \setone_{\lbrace u > 0 \rbrace}
		\qquad\text{and}\qquad
		v|_{\partial\Omega} = u|_{\partial\Omega} \; \setone_{\lbrace u|_{\partial\Omega} > 0 \rbrace}
	\]
	and hence
	\[
		0 \ge \int_\Omega fv + \int_{\partial\Omega} g v 
			= \lambda \int_\Omega u v + a_\beta(u,v) \\
			= \lambda \int_\Omega |v|^2 + a_\beta(v,v)
			\ge 0
	\]
	by~\eqref{eq:laxmilgram}. By~\eqref{eq:formelliptic} this shows that $v = 0$, i.e., $u \le 0$ almost everywhere.
	We have shown that the resolvent $(\lambda - A_2)^{-1}$ is a positive operator.
	Since every positive operator is continuous~\cite{AN09} we deduce that $\lambda - A_2$ is in fact invertible.

	In particular we have proved that $A_2$ is closed. Hence $D(A_2)$ is a Banach space for the graph norm of $A_2$,
	and by definition of $A_2$ we have $D(A_2) \subset H^1(\Omega) \times \{0\}$.
	Since both of these spaces are continuously embedded into $L^2(\Omega) \times L^2(\partial\Omega)$,
	we deduce from the closed graph theorem that $D(A_2)$ is continuously embedded into $H^1(\Omega) \times \{0\}$.
\end{proof}

We always equip $D(A_2)$ with the graph norm.

Now we can define mild and classical solutions of $(P_{u_0,f,g})$.
The definition of a \emph{classical solution} is obtained
by writing $(P_{u_0,f,g})$ in terms of $A_2$ in a straight-forward way, assuming smoothness
in the time variable. The definition of a \emph{mild solution} is similar, but uses an
integrated form of the equation. These two notions are the most common ones in the study of abstract Cauchy problems.
\begin{definition}\label{def:semisol}
	Let $I = [0,T]$ for some $T > 0$, or let $I = [0,\infty)$.
	\begin{enumerate}[(a)]
	\item
		We say that a function $u$ is a \emph{classical $L^2$-solution of $(P_{u_0,f,g})$ on $I$} if $u$ is in
		$\mathrm{C}^1(I;L^2(\Omega))$, we have $u(0) = u_0$, the mapping $t \mapsto (u(t),0)$ is in $\mathrm{C}(I;D(A_2))$
		and the relation
		\begin{equation}\label{eq:classical}
			(u_t(t), 0) - A_2(u(t),0) = (f(t),g(t))
		\end{equation}
		holds for all $t \in I$.
	\item
		We say that a function $u$ is a \emph{mild $L^2$-solution of $(P_{u_0,f,g})$ on $I$} if $u \in \mathrm{C}(I;L^2(\Omega))$,
		$(\int_0^t u(s), 0) \in D(A_2)$ for all $t \ge 0$ and
		\begin{equation}\label{eq:mild}
			(u(t) - u_0, 0) - A_2 \Bigl( \int_0^t u(s), 0 \Bigr) = \Bigl(\int_0^t f(s), \int_0^t g(s) \Bigr)
		\end{equation}
		for all $t \ge 0$.
	\end{enumerate}
\end{definition}

It will turn out later that weak solutions and mild $L^2$-solutions are in fact the same.
Let us start with an easy relationship between the three notions of a solution.
\begin{theorem}\label{thm:weakersol}
	Let either $I = [0,T]$ with $T > 0$ or $I = [0,\infty)$.
	\begin{enumerate}[(a)]
	\item
		Every classical $L^2$-solution of $(P_{u_0,f,g})$ on $I$ is a weak solution on $I$.
	\item
		Every weak solution of $(P_{u_0,f,g})$ on $I$ is a mild $L^2$-solution on $I$.
	\end{enumerate}
\end{theorem}
\begin{proof}
	All three definitions depend only on the behavior of $u$ on bounded intervals, so
	it suffices to consider the case $I = [0,T]$.
	\begin{enumerate}[(a)]
	\item
		Let $u$ be a classical $L^2$-solution. Then $u \in \mathrm{C}([0,T];H^1(\Omega))$
		by Lemma~\ref{lem:DA2}, which shows that $u$ has the regularity requested in Definition~\ref{def:weaksol}.
		Let $\psi$ be in $H^1(0,T;H^1(\Omega))$ and satisfy $\psi(T) = 0$. From~\eqref{eq:classical}
		and Remark~\ref{rem:formform} we obtain that
		\[
			\int_\Omega u_t(t) \, \psi(t) + a_\beta(u(t),\psi(t)) = \int_\Omega f(t) \, \psi(t) + \int_{\partial\Omega} g(t) \, \psi(t)
		\]
		for all $t \in [0,T]$. Integrating over $[0,T]$ and integrating the first summand by parts
		this gives~\eqref{eq:weak}.
	\item
		Let $u$ be a weak solution. Fix functions $\phi \in H^1(0,T)$ and $\eta \in H^1(\Omega)$, where $\phi(T) = 0$.
		Define $\psi(t) \coloneqq \phi(t) \cdot \eta$. Then $\psi \in H^1(0,T;H^1(\Omega))$
		with $\psi(T) = 0$ and hence
		\begin{align*}
			& -\int_0^T \Bigl( \int_\Omega u(s) \eta \Bigr) \, \phi_t(s) \\
				& \qquad = \Bigl( \int_\Omega u_0 \eta \Bigr) \, \phi(0)
				+ \int_0^T \Bigl( -a_\beta(u(s),\eta) + \int_\Omega f(s) \eta + \int_{\partial\Omega} g(s) \eta \Bigr) \, \phi(s)
		\end{align*}
		by~\eqref{eq:weak}. Hence $t \mapsto \int_\Omega u(t)\eta$ is weakly differentiable for all $\eta \in H^1(\Omega)$ with weak derivative
		\[
			\frac{\dx}{\dx t} \int_\Omega u(t) \eta
				= -a_\beta(u(s),\eta) + \int_\Omega f(s)\eta + \int_{\partial\Omega} g(t) \eta.
		\]
		and initial value $\int_\Omega u(0) \eta = \int_\Omega u_0 \eta$, hence $u(0) = u_0$.
		We deduce that
		\[
			\int_\Omega u(t) \eta
				= \int_\Omega u_0 \eta + \int_0^t \Bigl( -a_\beta(u(s),\eta) + \int_\Omega f(s) \eta + \int_{\partial\Omega} g(s) \eta \Bigr) \\
		\]
		for all $t \in [0,T]$ and all $\eta \in H^1(\Omega)$. Since $u \in L^2(0,T;H^1(\Omega))$ and $v \mapsto a_\beta(v,\eta)$
		is a continuous linear functional on $H^1(\Omega)$, this implies that
		\[
			\int_\Omega (u(t) - u_0) \eta + a_\beta\Bigl( \int_0^t u(s), \eta\Bigr)
				= \int_\Omega \Bigl( \int_0^t f(s) \Bigr) \eta + \int_{\partial\Omega} \Bigl( \int_0^t g(s) \Bigr) \eta
		\]
		for all $\eta \in H^1(\Omega)$.
		Hence by Remark~\ref{rem:formform} the function $u$ is a weak solution.
	\qedhere
	\end{enumerate}
\end{proof}

We want to establish the existence of a weak solution
via the theory of resolvent positive operators.
Since $L^2(\Omega) \times L^2(\partial\Omega)$ is a Banach lattice with
order continuous norm, the resolvent positive operator $A_2$ generates a once integrated
semigroup~\cite[Theorem~3.11.7]{ABHN01}. This yields the following
existence, uniqueness and comparison results for $L^2$-solutions.

\begin{proposition}\label{prop:respossol}
	Let $u_0 \in L^2(\Omega)$, $f \in L^2(0,T; L^2(\Omega))$ and $g \in L^2(0,T; L^2(\partial\Omega))$
	for some $T > 0$.
	\begin{enumerate}[(a)]
	\item\label{ass:uniquemild}
		Problem $(P_{u_0,f,g})$ has at most one mild $L^2$-solution.
	\item\label{ass:exclass}
		Assume that $u_0 \in L^2(\Omega)$ satisfies $Au_0 \in L^2(\Omega)$ and $\frac{\partial u_0}{\partial \nu_A} \in L^2(\partial\Omega)$
		and that $f \in \mathrm{C}^2( [0,T]; L^2(\Omega) )$ and $g \in \mathrm{C}^2( [0,T]; L^2(\partial\Omega) )$.
		If $\frac{\partial u_0}{\partial \nu_A} + \beta u_0 = g(0)$ holds
		and $v \coloneqq Au_0 + f(0) \in L^2(\Omega)$ satisfies $Av \in L^2(\Omega)$ and $\frac{\partial v}{\partial \nu_A} \in L^2(\partial\Omega)$,
		then $(P_{u_0,f,g})$ has a classical $L^2$-solution.
	\item\label{ass:comp}
		Assume that $u_0 \ge 0$, $f(t) \ge 0$ and $g(t) \ge 0$ almost everywhere for almost every $t \in (0,T)$.
		If $u$ is a mild $L^2$-solution of $(P_{u_0,f,g})$, then $u(t) \ge 0$ almost everywhere for every $t \in (0,T)$.
	\end{enumerate}
\end{proposition}
\begin{proof}
	By Definition~\ref{def:semisol} a function $u$ is a mild (resp.: classical) $L^2$-solution of $(P_{u_0,f,g})$
	if and only if the mapping $t \mapsto (u(t),0)$ is a mild (resp.: classical) solution of
	the abstract Cauchy problem associated with $A_2$ with inhomogeneity $(f,g)$, confer~\cite[\S 3.1]{ABHN01}.
	Hence part~\eqref{ass:comp} follows from~\cite[Theorem~3.11.11]{ABHN01}. This implies in particular that
	$u = 0$ is the unique mild $L^2$-solution if $u_0 = 0$, $f=0$ and $g = 0$,
	so part~\eqref{ass:uniquemild} follows from the linearity of the equation.
	Finally, the conditions on $u_0$ in part~\eqref{ass:exclass} can be rephrased by saying that
	\[
		(u_0, 0) \in D(A_2)
		\quad\text{and}\quad
		A_2(u_0,0) + (f(0),g(0)) \in D(A_2).
	\]
	Hence the existence of a classical $L^2$-solutions follows from~\cite[Corollary~3.2.11]{ABHN01}.
\end{proof}

We want to show that for all square-integrable functions $u_0$, $f$ and $g$ we have a unique weak solution.
As a first step we prove a bound for classical $L^2$-solutions in the norm of
$\mathrm{C}( [0,T]; L^2(\Omega) ) \cap L^2(0,T; H^1(\Omega) )$.
\begin{lemma}\label{lem:H1est}
	If $u$ is a classical $L^2$-solution of $(P_{u_0,f,g})$ on $[0,T]$ for some $T > 0$, then
	\begin{equation}\label{eq:H1est}
		\sup_{0 \le t \le T} \int_\Omega |u(t)|^2 + \int_0^T \int_\Omega |\nabla u|^2 \le c \int_\Omega |u_0|^2 + c \int_0^T \int_\Omega |f(t)|^2 + c \int_0^T \int_{\partial\Omega} |g(t)|^2
	\end{equation}
	for a constant $c \ge 0$ that depends only on $T$, $\Omega$ and the values $\mu$ and $\omega$ in~\eqref{eq:aelliptic}.
\end{lemma}
\begin{proof}
	Let $t \in [0,T]$ be arbitrary. Then
	\begin{align*}
		& \frac{1}{2} \int_\Omega |u(t)|^2 - \frac{1}{2} \int_\Omega |u_0|^2
			= \frac{1}{2} \int_0^t \frac{\dx}{\dx s} \int_\Omega |u(s)|^2
			= \int_0^t \int_\Omega u(s) \; u_t(s) \\
			& \qquad = \int_0^t \int_\Omega u(s) \bigl( Au(s) + f(s) \bigr) \\
			& \qquad = \int_0^t \int_{\partial\Omega} \frac{\partial u(s)}{\partial \nu_A} \, u(s) - \int_0^t a_0(u(s), u(s)) + \int_0^t \int_\Omega f(s) \, u(s) \\
			& \qquad = \int_0^t \int_\Omega f(s) \, u(s) + \int_0^t \int_{\partial\Omega} g(s) \, u(s) - \int_0^t a_\beta(u(s), u(s)) \\
			& \qquad \le \frac{1}{2} \int_0^t \int_\Omega |f(s)|^2 + \frac{1}{4\eps} \int_0^t \int_{\partial\Omega} |g(s)|^2 \\
				& \qquad\qquad - \Bigl( \frac{\mu}{2} - \eps c_1^2) \int_0^t \int_\Omega |\nabla u(s)|^2 + (\omega + \tfrac{1}{2} + \eps c_1^2) \int_0^t \int_\Omega |u(s)|^2,
	\end{align*}
	where we have used Young's inequality and~\eqref{eq:aelliptic}. Here $c_1 \ge 0$ is the norm of the trace
	operator from $H^1(\Omega)$ to $L^2(\partial\Omega)$. We pick $\eps \coloneqq \frac{\mu}{4c_1^2}$ and vary over $t$
	to deduce that
	\begin{align*}
		& \sup_{0 \le s \le t} \int_\Omega |u(s)|^2 + \int_0^t \int_\Omega |\nabla u(s)|^2 \\
			& \qquad \le c_2 \int_\Omega |u_0|^2 + c_2 \int_0^t \int_\Omega |f(s)|^2 + c_2 \int_0^t \int_{\partial\Omega} |g(s)|^2 + c_2 \int_0^t \int_\Omega |u(s)|^2 \\
			& \qquad \le c_2 \int_\Omega |u_0|^2 + c_2 \int_0^t \int_\Omega |f(s)|^2 + c_2 \int_0^t \int_{\partial\Omega} |g(s)|^2 + t c_2 \sup_{0 \le s \le t} \int_\Omega |u(s)|^2
	\end{align*}
	for all $t \in [0,T]$ with a constant $c_2 \ge 0$ that depends only on $c_1$, $\mu$ and $\omega$.
	This shows that with $t_0 \coloneqq \frac{1}{2c_2}$ we have
	\begin{align*}
		& \sup_{0 \le s \le t} \int_\Omega |u(s)|^2 + \int_0^t \int_\Omega |\nabla u(s)|^2 \\
			& \qquad\qquad \le 2c_2 \int_\Omega |u_0|^2 + 2c_2 \int_0^t \int_\Omega |f(s)|^2 + 2c_2 \int_0^t \int_{\partial\Omega} |g(s)|^2
	\end{align*}
	for all $t \in [0,t_0]$.
	We split $[0,T]$ into finitely many intervals of length at most $s_0$ and apply the last inequality
	successively on these intervals. This gives~\eqref{eq:H1est}.
\end{proof}

We also collect some results about the homogeneous problem $(P_{u_0,0,0})$ for later use.
To this end we introduce the generator $A_{2,h}$ for the homogeneous problem, which is the part of
$A_2$ in $L^2(\Omega) \times \{0\}$. All of the following results stem from semigroup theory.
\begin{proposition}\label{prop:semigroup}
	The operator $A_{2,h}$ given by
	\begin{align*}
		D(A_{2,h}) & = \Bigl\{ u \in H^1(\Omega) : Au \in L^2(\Omega), \; \frac{\partial u}{\partial \nu_A} + \beta u = 0 \Bigr\} \\
		A_{2,h} u & = Au
	\end{align*}
	is the generator of an analytic $\mathrm{C}_0$-semigroup $(T_{2,h}(t))_{t \ge 0}$ on $L^2(\Omega)$.
	Given $u_0 \in L^2(\Omega)$, the function $u$ defined by $u(t) \coloneqq T_{2,h}(t)u_0$
	is the unique mild $L^2$-solution of $(P_{u_0,0,0})$, and we have the following properties:
	\begin{enumerate}[(i)]
	\item\label{it:semigroup_Linfbdd}
		There exist $M \ge 0$ and $\omega \in \mathds{R}$ depending only on $N$, $\Omega$
		and the coefficients of the equation such that $\|u(t)\|_{L^\infty(\Omega)} \le M \e^{\omega t} \|u_0\|_{L^\infty(\Omega)}$
		for all $t \ge 0$.
	\item\label{it:semigroup_hoelder}
		For every $t > 0$ we have $u(t) \in \mathrm{C}(\overline{\Omega})$.
	\item\label{it:semigroup_Csc}
		If $u_0 \in \mathrm{C}(\overline{\Omega})$, then $u \in \mathrm{C}([0,\infty);\mathrm{C}(\overline{\Omega}))$ for all $T > 0$.
	\end{enumerate}
\end{proposition}
\begin{proof}
	The operator $-A_{2,h}$ is associated with the bounded, $L^2(\Omega)$-elliptic bilinear form $a_\beta\colon H^1(\Omega) \times H^1(\Omega) \to \mathds{R}$
	defined in~\eqref{eq:defab}. Hence $A_{2,h}$ generates an analytic $\mathrm{C}_0$-semigroup on $L^2(\Omega)$, see~\cite[Proposition~XVII.A.6.3]{DL5}.
%	By~\cite{...} this semigroup is positive.
	By construction a function $u$ is a mild solution for the abstract Cauchy problem associated with $A_{2,h}$ if and only if
	it is a mild $L^2$-solution of $(P_{u_0,0,0})$, which proves the assertion about the mild $L^2$-solutions~\cite[Theorem~3.1.12]{ABHN01}.
	Property~\eqref{it:semigroup_Linfbdd} follows from~\cite[Proposition~7.1]{Dan00}. Properties~\eqref{it:semigroup_hoelder}
	and~\eqref{it:semigroup_Csc} have been proved in~\cite[Theorem~4.3]{Nit11} for bounded coefficients. The same arguments
	work here, but compare also~\cite{Diss,Nitt12}, where unbounded (and nonlinear) coefficients are considered.
\end{proof}

The following is our main existence theorem.
\begin{theorem}\label{thm:mainex}
	Let $u_0 \in L^2(\Omega)$, $f \in L^2(0,T;L^2(\Omega))$ and $g \in L^2(0,T; L^2(\partial\Omega) )$ be given,
	where $T > 0$ is arbitrary.
	Then there exists a weak solution $u$ of $(P_{u_0,f,g})$ on $[0,T]$, which is unique even within the class
	of mild $L^2$-solutions.
\end{theorem}
\begin{proof}
	Pick sequences $(f_n) \subset \mathrm{C}^2( [0,T]; L^2(\Omega) )$ and
	$(g_n) \subset \mathrm{C}^2( [0,T]; L^2(\partial\Omega) )$ that satisfy $f_n(0) = 0$, $g_n(0) = 0$,
	$f_n \to f$ in $L^2(0,T;L^2(\Omega))$
	and $g_n \to g$ in $L^2(0,T; L^2(\partial\Omega) )$. Since $A_{2,h}$ is the
	generator of a $\mathrm{C}_0$-semigroup,
	there exists a sequence $(u_{n,0}) \subset D(A_{2,h}^2)$
	satisfying $u_{n,0} \to u_0$ in $L^2(\Omega)$, see~\cite[Proposition~II.1.8]{EN00}.
	By Proposition~\ref{prop:respossol} there exists a classical $L^2$-solutions $u_n$ of
	$(P_{u_{n,0},f_n,g_n})$.

	By Lemma~\ref{lem:H1est} the sequence
	$(u_n)$ is Cauchy in $\mathrm{C}( [0,T]; L^2(\Omega) ) \cap L^2(0,T; H^1(\Omega) )$.
	Denote its limit by $u$. Using that $u_n$ is a weak solution of $(P_{u_{n,0},f_n,g_n})$
	by Theorem~\ref{thm:weakersol}, we can pass in~\eqref{eq:weak}
	to the limit and obtain that
	$u$ is a weak solution of $(P_{u_0,f,g})$.
	Uniqueness has already been asserted in Proposition~\ref{prop:respossol}.
\end{proof}

Since being a solution is a local concept, we obtain the following corollary.
\begin{corollary}\label{cor:mainex}
	For given functions $u_0 \in L^2(\Omega)$, $f \in L^2_{\mathrm{loc}}( [0,\infty); L^2(\Omega) )$
	and $g \in L^2_{\mathrm{loc}}( [0,\infty); L^2(\partial\Omega) )$,
	equation $(P_{u_0,f,g})$ has a weak solution on $[0,\infty)$, which is unique
	even in the class of mild solutions.
\end{corollary}

We deduce the following from Theorem~\ref{thm:weakersol} and
Theorem~\ref{thm:mainex} or Corollary~\ref{cor:mainex}, respectively.
\begin{corollary}
	For problem $(P_{u_0,f,g})$ the notions of weak and mild solutions coincide.
\end{corollary}

We have seen that Problem $(P_{u_0,f,g})$ admits unique weak solutions. One might expect that
this implies that $A_2$ is the generator of a strongly continuous semigroup. Obviously, this
is not true since $A_2$ is not densely defined. Even worse, the operator does not even
satisfy Hille-Yosida estimates as the following example shows.
\begin{example}\label{ex:hille}
	Set $\Omega = (0,1)$ and consider the Laplace operator with Neumann boundary conditions,
	i.e., $A_2(u,0) \coloneqq (u'', (u'(0), -u'(1)))$ on $L^2(0,1) \times \mathds{R}^2$.
	For $\lambda > 0$ we can explicitly calculate that
	\[
		u_\lambda \coloneqq (\lambda - A)^{-1}(0,(0,1))
	\]
	is given by the formula
	\[
		u_\lambda(x) = \frac{\exp(\sqrt{\lambda}x) + \exp(-\sqrt{\lambda}x)}{\sqrt{\lambda} \bigl(\exp(\sqrt{\lambda})-\exp(-\sqrt{\lambda})\bigr)},
	\]
	from which we obtain after some calculations that
	\[
		\|u_\lambda\|_{L^2(\Omega)} \sim \tfrac{1}{\sqrt{2}} \lambda^{-\frac{3}{4}}
	\]
	as $\lambda \to \infty$.
	Hence $\|\lambda R(\lambda,A)\| \sim c \lambda^{\frac{1}{4}}$ as $\lambda \to \infty$
	for some constant $c > 0$, which shows that $A_2$ is not a Hille-Yosida operator in
	the sense of~\cite[\S 3.5]{ABHN01}. This was already clear since every Hille-Yosida operator
	on reflexive space is densely defined~\cite[Proposition~3.3.8]{ABHN01}.
\end{example}

\section{Regularity}\label{sec:reg}

The goal of this section is to show that for $u_0 \in \mathrm{C}(\overline{\Omega})$ the weak solution
of $(P_{u_0,f,g})$ is continuous on the parabolic cylinder $[0,\infty) \times \overline{\Omega}$,
so in particular continuous up to the boundary.
The main tool is the following pointwise a priori estimate, which we will use also for the study
of the asymptotic behavior.

\begin{proposition}\label{prop:Linfbound}
	Fix $T > 0$. Let $r_1, r_2, q_1, q_2 \in [2, \infty)$ satisfy
	\begin{equation}\label{eq:qrrel}
		\frac{1}{r_1} + \frac{N}{2q_1} < 1
		\qquad\text{and}\qquad
		\frac{1}{r_2} + \frac{N-1}{2q_2} < \frac{1}{2}.
	\end{equation}
	Let $u_0 \in L^2(\Omega)$, $f \in L^{r_1}(0,T; L^{q_1}(\Omega) )$ and $g \in L^{r_2}(0,T; L^{q_2}(\partial\Omega) )$ be given
	and denote by $u$ the weak solution of $(P_{u_0,f,g})$.
	Then
	\begin{equation}\label{eq:Linfbound}
		\|u\|_{L^\infty(\frac{T}{2},T; L^\infty(\Omega)}^2
%		\esssup_{(t,x) \in [\frac{T}{2},T] \times \Omega} |u(t)|^2
			\le c \|u\|_{L^2(0,T; L^2(\Omega)}^2
				+ c \|f\|_{L^{r_1}(0,T;L^{q_1}(\Omega))}^2
				+ c \|g\|_{L^{r_2}(0,T;L^{q_2}(\partial\Omega))}^2,
	\end{equation}
	where $c$ depends only on $T$, $N$, $\Omega$, $r_1$, $q_1$, $r_2$, $q_2$ and the coefficients of the equation.

	If we have $u_0 = 0$, then we obtain the global estimate
	\begin{equation}\label{eq:Linfboundglob}
		\|u\|_{L^\infty(0,T; L^\infty(\Omega)}^2
%		\esssup_{(t,x) \in [0,T] \times \Omega} |u(t)|^2
			\le c \|u\|_{L^2(0,T; L^2(\Omega)}^2
				+ c \|f\|_{L^{r_1}(0,T_0;L^{q_1}(\Omega))}^2
				+ c \|g\|_{L^{r_2}(0,T_0;L^{q_2}(\partial\Omega))}^2.
	\end{equation}
\end{proposition}

The proof of Proposition~\ref{prop:Linfbound} is lengthy and technical. We postpone it to Appendix~\ref{app:pointest}
in order not to interrupt the train of thought.
We will use mainly the following consequence of Proposition~\ref{prop:Linfbound},
which arises from combining it with Proposition~\ref{prop:semigroup}.
\begin{theorem}\label{thm:inLinf}
	Let $T > 0$ be arbitrary, let $f$ and $g$ satisfy the conditions of Proposition~\ref{prop:Linfbound}
	and let $u_0 \in L^\infty(\Omega)$ be given. Then the weak solution $u$ of $(P_{u_0,f,g})$
	satisfies
	\begin{equation}\label{eq:inLinf}
		\begin{aligned}
			\|u\|_{L^\infty(0,T; L^\infty(\Omega))}
				\le c \|u_0\|_{L^\infty(\Omega)}
				+ c \|f\|_{L^{r_1}(0,T;L^{q_1}(\Omega))} + c \|g\|_{L^{r_2}(0,T;L^{q_2}(\Omega))},
		\end{aligned}
	\end{equation}
	where $c$ depends on the same parameters as in Proposition~\ref{prop:Linfbound}.
\end{theorem}
\begin{proof}
	By linearity we have
	$u(t) = T_{2,h}(t)u_0 + v(t)$, where $(T_{2,h}(t))_{t \ge 0}$ has been introduced in Proposition~\ref{prop:semigroup}
	and $v$ is the weak solution of $(P_{0,f,g})$.
	Hence we deduce from~\eqref{eq:Linfboundglob} and Proposition~\ref{prop:semigroup} that
	\begin{align*}
		\|u\|_{L^\infty(0,T; L^\infty(\Omega))}^2
			& \le 2 \sup_{0 \le t \le T} \|T_{2,h}(t)u_0\|_{L^\infty(\Omega)}^2
				+ 2 \|v\|_{L^\infty(0,T; L^\infty(\Omega))}^2 \\
			& \le 2M^2 \e^{2|\omega| T} \|u_0\|_{L^\infty(\Omega)}^2
				+ 2c \, \|v\|_{L^2(0,T; L^2(\Omega)}^2 \\
				& \qquad + 2c \, \|f\|_{L^{r_1}(0,T;L^{q_1}(\Omega))}^2 + 2c \, \|g\|_{L^{r_2}(0,T;L^{q_2}(\Omega))}^2.
	\end{align*}
	In addition, by Lemma~\ref{lem:H1est} and H\"older's inequality we have
	\[
		\|v(s)\|_{L^2(\Omega)}^2
				\le c \|u_0\|_{L^\infty(\Omega)}^2 + c \|f\|_{L^{r_1}(0,T;L^{q_1}(\Omega))}^2 + c \|g\|_{L^{r_2}(0,T;L^{q_2}(\Omega))}^2
	\]
	for all $s \in [0,T]$, where we note that by the proof of Theorem~\ref{thm:mainex} the lemma
	is valid for all weak solutions, not only classical solutions.
	Combining these two estimates we have proved~\eqref{eq:inLinf}.
\end{proof}

We use~\eqref{eq:inLinf} to deduce continuity of the solution up to the boundary of the parabolic cylinder,
which is our main regularity result.

\begin{theorem}\label{thm:inhomcont}
	Let $T > 0$ be arbitrary, let $f$ and $g$ satisfy the conditions of Proposition~\ref{prop:Linfbound}
	and let $u_0 \in \mathrm{C}(\overline{\Omega})$ be given. Then the weak solution $u$ of $(P_{u_0,f,g})$
	is in $\mathrm{C}([0,T]; \mathrm{C}(\overline{\Omega}))$. So in particular $u(t) \to u_0$ uniformly
	on $\overline{\Omega}$ as $t \to 0$.
\end{theorem}
\begin{proof}
	Let $A_X$ denote the realization of $A$ in $X \coloneqq L^{q_1}(\Omega) \times L^{q_2}(\partial\Omega)$
	with the same boundary conditions as $A_2$, i.e.,
	\begin{align*}
		D(A_X) & \coloneqq \Bigl\{ (u,0) \in D(A_2) : Au \in L^{q_1}(\Omega), \; \frac{\partial u}{\partial \nu_A} \in L^{q_2}(\partial\Omega) \Bigr\} \\
		A_X(u,0) & \coloneqq \Bigl( Au, \; -\frac{\partial u}{\partial \nu_A} - \beta u|_{\partial\Omega} \Bigr).
	\end{align*}
	Thus $(u,0) \in D(A_X)$ if and only if there exist $f \in L^{q_1}(\Omega)$ and $g \in L^{q_2}(\partial\Omega)$
	such that $u$ solves
	\[
		\left\{ \begin{aligned}
			Au & = f && \text{on } \Omega \\
			\frac{\partial u}{\partial \nu_A} + \beta u & = g && \text{on } \partial\Omega
		\end{aligned} \right.
	\]
	in the weak sense. Since by~\eqref{eq:qrrel} we have in particular $q_1 > \frac{N}{2}$ and $q_2 > \frac{N-1}{2}$,
	elliptic regularity theory shows that in this case $u \in \mathrm{C}(\overline{\Omega})$,
	compare~\cite[Theorem~3.14]{Nit11} for bounded coefficients or~\cite[Example~4.2.7]{Diss} for the general case.
	Hence $D(A_X) \subset \mathrm{C}(\overline{\Omega}) \times \{0\}$ and in particular $D(A_X) \subset X$.
	Hence $A_X$ is the part of the resolvent positive operator $A_2$ in $X$, and hence is resolvent positive.
	Thus $A_X$ generates a once integrated semigroup on $X$ by~\cite[Theorem~3.11.7]{ABHN01}.

	Pick sequences $(f_n) \subset \mathrm{C}^2( [0,T]; L^\infty(\Omega) )$ and
	$(g_n) \subset \mathrm{C}^2( [0,T]; L^\infty(\partial\Omega) )$
	that satisfy $f_n(0) = 0$, $g_n(0) = 0$, $f_n \to f$ in $L^{r_1}(0,T; L^{q_1}(\Omega))$ and
	$g_n \to g$ in $L^{r_2}(0,T; L^{q_2}(\partial\Omega))$, and let $v_n$ denote the weak solution of $(P_{0,f_n,g_n})$.

	By~\cite[Corollary~3.2.11]{ABHN01} the abstract Cauchy problem
	\[
		\left\{ \begin{aligned}
			\dot{W}_n(t) & = A_XW_n(t) + (f_n(t), g_n(t)) \\
			W(0) & = (0,0)
		\end{aligned} \right.
	\]
	has a unique solution $W_n = (w_n, 0) \in \mathrm{C}^1([0,T]; X) \cap \mathrm{C}([0,T]; D(A_X))$, and
	in particular we have $w_n \in \mathrm{C}([0,T]; \mathrm{C}(\overline{\Omega}))$;
	we could call $w_n$ a \emph{classical $X$-solution of $(P_{0,f_n,g_n})$} in analogy to Definition~\ref{def:semisol}. 
	The function $w_n$ is in particular a classical $L^2$-solution of~\eqref{def:semisol}, hence $w_n = v_n$
	by uniqueness. We have shown that $v_n \in \mathrm{C}([0,T]; \mathrm{C}(\overline{\Omega}))$.

	Now, since by Theorem~\ref{thm:inLinf} we have $v_n \to v$ uniformly on $[0,T] \times \overline{\Omega}$,
	where $v$ denotes the weak solution of $(P_{0,f,g})$,
	we deduce that $v \in \mathrm{C}([0,T]; \mathrm{C}(\overline{\Omega}))$.
	Hence, since $u(t) = T_{2,h}(t)u_0 + v(t)$ with $(T_{2,h}(t))_{t \ge 0}$ defined in Proposition~\ref{prop:semigroup},
	continuity of $u$ follows from Proposition~\ref{prop:semigroup}.
\end{proof}

\begin{remark}\label{rem:latercont}
	If in Theorem~\ref{thm:inhomcont} we only have $u_0 \in L^2(\Omega)$ instead of $u_0 \in \mathrm{C}(\overline{\Omega})$,
	we still obtain that $u|_{[t_0,T]} \in \mathrm{C}([t_0,T]; \mathrm{C}(\overline{\Omega}))$ for all $t_0 \in (0,T)$.
	In fact, this can be seen easily from the proof since by Proposition~\ref{prop:semigroup}
	$t \mapsto T_{2,h}(t)u_0$ is continuous from $[t_0,\infty)$ to $\mathrm{C}(\overline{\Omega})$ for every $t_0 > 0$.

	In particular, $u_0 \in \mathrm{C}(\overline{\Omega})$ is a necessary condition for the convergence
	$u(t) \to u_0$ as $t \to 0$ to be uniform on $\overline{\Omega}$.
	Theorem~\ref{thm:inhomcont} shows that it is also sufficient if $f$ and $g$ do not behave too badly.
\end{remark}

We close this section by a comparison with the situation for Dirichlet boundary conditions.
\begin{remark}
	For the Dirichlet initial-boundary value problem studied in~\cite{Are00} one has to work with a
	realization $A_{c,D}$ of $A$ with Dirichlet boundary conditions in a space of continuous functions because
	$L^p$-regularity conditions on the boundary data do not suffice in order to obtain
	continuous solutions, which contrasts the situation in Theorem~\ref{thm:inhomcont} for Neumann boundary data.
	This leads to a minor difficulty. More precisely,
	since $\mathrm{C}(\partial\Omega)$ does not have order continuous norm, it is not immediately clear that $A_{c,D}$
	is the generator of a once integrated semigroup. In fact, this is even false since if $A_{c,D}$ were the generator of a once integrated
	semigroup, then by~\cite[Corollary~3.2.11]{ABHN01} there would exist a mild solution of the corresponding abstract Cauchy problem
	\[
		\left\{ \begin{aligned}
			u_t(t) & = \Delta u(t) \\
			u(t)|_{\partial\Omega} & = \phi(t) \\
			u(0) & = u_0
		\end{aligned} \right.
	\]
	regardless of any compatibility assumptions between $\phi \in \mathrm{C}^1([0,\infty); \mathrm{C}(\partial\Omega))$
	and $u_0 \in \mathrm{C}(\overline{\Omega})$.
	This contradicts the simple observation that the existence of a mild solution enforces the condition $\phi(0) = u_0|_{\partial\Omega}$,
	see~\cite[Proposition~3.2]{Are00}.
	Still, $A_{c,D}$ generates a twice integrated semigroup~\cite[Theorem~3.11.5]{ABHN01}, which is sufficient
	for the results in~\cite{Are00}.

	The situation is different for Neumann boundary conditions, as we can already expect from the fact that
	no compatibility condition appears in Theorem~\ref{thm:inhomcont}. In fact, we have a once integrated semigroup
	in that case. In order to see this, consider the realization $A_c$
	in $\mathrm{C}(\overline{\Omega}) \times \mathrm{C}(\partial\Omega)$ of $A$ with Robin boundary conditions
	and set $Z \coloneqq \mathrm{C}(\overline{\Omega}) \times \{0\}$. Then $D(A_c) \subset Z$, the space $Z$ is invariant
	under the resolvent of $A_c$ and the part of $A_c$ in $Z$ is the generator of a strongly continuous
	semigroup~\cite[Theorem~4.3]{Nit11}. Hence by~\cite[Theorem~3.10.4]{ABHN01} the operator $A_c$ generates
	a once integrated semigroup on $\mathrm{C}(\overline{\Omega}) \times \mathrm{C}(\partial\Omega)$.

	It can be seen from Example~\ref{ex:hille} that the operator $A_c$ fails to be a Hille-Yosida operator.
	In this respect, the situation is the same as for Dirichlet boundary conditions~\cite[Remark~2.5~b)]{Are00}.
\end{remark}

\section{Convergence}\label{sec:conv}

In this section we study boundedness of the solution $u$ of $(P_{u_0,f,g})$ as $t \to \infty$. We are not interested
in (exponential) blow-up or decay, but want to consider the border case only. Inspired by our model case, i.e.,
$A = \Delta$ and $\beta = 0$, a natural condition that helps with this issue is to assume conservation of total energy, i.e.,
\begin{equation}\label{eq:preserve}
	\int_\Omega u(t) = \int_\Omega u_0 + \int_0^t \int_\Omega f(s) + \int_0^t \int_{\partial\Omega} g(s)
\end{equation}
for all $t > 0$. We restrict ourselves to this situation, which can be characterized as follows.

\begin{proposition}\label{prop:charpres}
	The following assertions are equivalent:
	\begin{enumerate}[(i)]
	\item\label{eq:charpres_all}
		for every $T > 0$, $f \in L^2(0,T;L^2(\Omega))$, $g \in L^2(0,T;L^2(\partial\Omega))$ and $u_0 \in L^2(\Omega)$
		relation~\eqref{eq:preserve} holds for all $t \in [0,T]$,
		where $u$ is the weak solution of~$(P_{u_0,f,g})$;
	\item\label{eq:charpres_zero}
		for every $u_0 \in L^2(\Omega)$ we have $\int_\Omega u(t) = \int_\Omega u_0$ for all $t > 0$,
		where $u$ is the weak solution of~$(P_{u_0,0,0})$;
	\item\label{eq:charpres_coeff}
		the relation
		\begin{equation}\label{eq:assc}
			\left\{ \begin{aligned}
				\sum\nolimits_{i=1}^N c_i & = d && \text{on } \Omega \\
				\sum\nolimits_{i=1}^N c_i \, \nu_i & = -\beta && \text{on } \partial\Omega
			\end{aligned} \right.
		\end{equation}
		holds in the weak sense, i.e.,
		\[
			\sum_{i=1}^N \int_\Omega c_i \, D_i\eta + \int_\Omega d \eta + \int_{\partial\Omega} \beta \eta = 0
			\quad \text{for all } \eta \in H^1(\Omega).
		\]
	\end{enumerate}
\end{proposition}
\begin{proof}
	Assume~\eqref{eq:charpres_coeff} and let $u$ be the weak solution of $(P_{u_0,f,g})$, which
	is a mild $L^2$-solution by Theorem~\ref{thm:weakersol}. By Remark~\ref{rem:formform} we have
	\begin{equation}\label{eq:charpres_weak}
		a_\beta\Bigl( \int_0^t u(s), v \Bigr) = \int_0^t \int_\Omega f(s) \, v + \int_0^t \int_{\partial\Omega} g(s) \, v - \int_\Omega (u(t) - u_0) \, v
	\end{equation}
	for all $v \in H^1(\Omega)$. Picking $v \coloneqq \setone_\Omega$ and using that by~\eqref{eq:assc} we have
	$a_\beta(\eta,\setone_\Omega) = 0$ for all $\eta \in H^1(\Omega)$ this gives~\eqref{eq:preserve}.

	It is trivial that~\eqref{eq:charpres_all} implies~\eqref{eq:charpres_zero}. So now assume that~\eqref{eq:charpres_zero} holds, i.e.,
	$\int_\Omega T_{2,h}(t)u_0 = \int_\Omega u_0$ for all $t \ge 0$ and all $u_0 \in L^2(\Omega)$,
	where $(T_{2,h}(t))_{t \ge 0}$ is defined in Proposition~\ref{prop:semigroup}. Then $\setone_\Omega$ is a fixed point
	of the adjoint semigroup $(T_{2,h}^\ast(t))_{t \ge 0}$, which implies $A_{2,h}^\ast \setone_{\Omega} = 0$, i.e.,
	$a_\beta(\eta, \setone_\Omega) = 0$ for all $\eta \in H^1(\Omega)$. This is~\eqref{eq:assc}.
\end{proof}

We aim towards a bound of the solution of $(P_{u_0,f,g})$ in $L^\infty(0,\infty; L^\infty(\Omega))$.
As a first step, we consider this problem only for the homogeneous problem $(P_{u_0,0,0})$, as we describe
in the following lemma.
\begin{lemma}\label{lem:Linfcontr}
	Under condition~\eqref{eq:assc} we have $\|u\|_{L^\infty(0,\infty; L^\infty(\Omega))} \le \|u_0\|_{L^\infty(\Omega)}$
	for the weak solution $u$ of $(P_{u_0,0,0})$ if and only if
	\begin{equation}\label{eq:assb}
		\left\{ \begin{aligned}
			\sum\nolimits_{j=1}^N b_j & = d && \text{on } \Omega \\
			\sum\nolimits_{j=1}^N b_j \, \nu_j & = -\beta && \text{on } \partial\Omega
		\end{aligned} \right.
	\end{equation}
	in the weak sense.
\end{lemma}
\begin{proof}
	Relation~\eqref{eq:assb} is equivalent to
	$a_\beta(\setone_\Omega, \eta) = 0$ for all $\eta \in H^1(\Omega)$,
	i.e., $A_{2,h} \setone_\Omega = 0$. Hence~\eqref{eq:assb} is equivalent to $\setone_\Omega$ being a fixed point of $(T_{2,h}(t))_{t \ge 0}$,
	where $(T_{2,h}(t))_{t \ge 0}$ is defined in Proposition~\ref{prop:semigroup}.

	Since $(T_{2,h}(t))_{t \ge 0}$ is positive, $T_{2,h}(t)\setone_\Omega = \setone_\Omega$ for all $t \ge 0$
	implies that the semigroup is contractive with respect to the norm of $L^\infty(\Omega)$, which is precisely the bound for $u$.
	On the other hand, if $(T_{2,h}(t))_{t \ge 0}$ is $L^\infty(\Omega)$-contractive and
	$\int_\Omega T_{2,h}(t)u_0 = \int_\Omega u_0$ for all $t \ge 0$, which is satisfied by Proposition~\ref{prop:charpres},
	then $\setone_\Omega$ is a fixed point of $(T_{2,h}(t))_{t \ge 0}$.
\end{proof}
We will see in Corollary~\ref{cor:asLinf} that~\eqref{eq:assb} implies that
also the inhomogeneous problem $(P_{u_0,f,g})$ has bounded solutions if we assume in addition that $\int_\Omega f(t) + \int_{\partial\Omega} g(t) = 0$
for all $t \ge 0$ and the functions $f$ and $g$ are not too irregular.
The first step into this direction is an $L^2$-bound on bounded time intervals, Proposition~\ref{prop:L2bound},
for which we need the following lemma.

\begin{lemma}\label{lem:coercive}
	If~\eqref{eq:assc} and~\eqref{eq:assb} hold, then
	$a_\beta(v,v) \ge \mu \int_\Omega |\nabla v|^2$ for all $v \in H^1(\Omega)$.
\end{lemma}
\begin{proof}
	By continuity of $a_\beta$ it suffices to prove the estimate for all $v \in H^1(\Omega) \cap L^\infty(\Omega)$.
	For such $v$ we have by~\eqref{eq:elliptic} and the chain rule that
	\begin{align*}
		a_\beta(v,v)
			& \ge \mu \int_\Omega |\nabla v|^2
				+ \frac{1}{2} \sum_{j=1}^N \int_\Omega b_j \, D_j(v^2)
				+ \frac{1}{2} \sum_{i=1}^N \int_\Omega c_i \, D_i(v^2)
				+ \int_\Omega d v^2
				+ \int_{\partial\Omega} \beta v^2 \\
			& = \mu \int_\Omega |\nabla v|^2,
	\end{align*}
	where in the second step we used the weak formulations of~\eqref{eq:assc} and~\eqref{eq:assb} with $\eta \coloneqq v^2 \in H^1(\Omega)$.
\end{proof}

\begin{proposition}\label{prop:L2bound}
	Let $u$ be the weak solution of $(P_{u_0,f,g})$ on $[0,T]$ for given $u_0 \in L^2(\Omega)$,
	$f \in L^2(0,T; L^2(\Omega))$ and $g \in L^2(0,T; L^2(\partial\Omega) )$.
	Assume that $\int_\Omega u_0 = 0$ and
	\begin{equation}\label{eq:fgintz}
		\int_\Omega f(t) + \int_{\partial\Omega} g(t) = 0 \quad \text{for all } t \ge 0.
	\end{equation}
	If~\eqref{eq:assc} and~\eqref{eq:assb} hold, then there exist $\tau > 0$ and $c \ge 0$ depending
	only on $\mu$ and $\Omega$ such that
	\begin{equation}\label{eq:expdecayest}
		\int_\Omega |u(t)|^2 \le \e^{-t/\tau} \int_\Omega |u_0|^2
			+ c \int_0^t \e^{(s-t)/\tau} \, \Bigl( \int_\Omega |f(s)|^2 + \int_{\partial\Omega} |g(s)|^2 \Bigr) \, \dx s
	\end{equation}
	for all $t \in [0,T]$.
\end{proposition}
\begin{proof}
	Since $u$ can be approximated by classical $L^2$-solutions of equations with right hand sides close
	to $f$ and $g$,
	compare the proof of Theorem~\ref{thm:mainex}, we can assume without loss of generality that $u$ is a
	classical $L^2$-solution of $(P_{u_0,f,g})$.

	By~\eqref{eq:fgintz} and Proposition~\ref{prop:charpres} we have $\int_\Omega u(t) = \int_\Omega u_0 = 0$ for all $t \in [0,T]$.
	Recall that $\Omega$ was assumed to be connected throughout the article.
	Hence by Poincar\'e's inequality and the Sobolev embedding theorems there exists $c_1 \ge 0$ depending only on $\Omega$ such that
	\begin{equation}\label{eq:poincare}
		\int_\Omega |u(t)|^2 + \int_{\partial\Omega} |u(t)|^2 \le c_1 \int_\Omega |\nabla u(t)|^2
	\end{equation}
	for all $t \ge 0$. Using Remark~\ref{rem:formform}, Lemma~\ref{lem:coercive}, Young's inequality and estimate~\eqref{eq:poincare} we obtain that
	\begin{align*}
		\frac{\dx}{\dx t} \frac{1}{2} \int_\Omega |u(t)|^2
			& = \int_\Omega u(t) \, u_t(t)
			= \int_\Omega u(t) \, \bigl( Au(t) + f(t) \bigr) \\
			& = \int_\Omega f(t) \, u(t) + \int_{\partial\Omega} g(t) \, u(t) - a_\beta(u(t),u(t)) \\
			& \le \frac{c_1}{2\mu} \Bigl( \int_\Omega |f(t)|^2 + \int_{\partial\Omega} |g(t)|^2 \Bigr)
				+ \frac{\mu}{2c_1} \Bigl( \int_\Omega |u(t)|^2 + \int_{\partial\Omega} |u(t)|^2 \Bigr) \\
				& \qquad - \mu \int_\Omega |\nabla u(t)|^2 \\
			& \le c_2 \Bigl( \int_\Omega |f(t)|^2 + \int_{\partial\Omega} |g(t)|^2 \Bigr)
				- \frac{\mu}{2} \int_\Omega |\nabla u(t)|^2,
	\end{align*}
	with $c_2 \coloneqq \frac{c_1}{2\mu}$.
	Define $\tau \coloneqq \frac{c_2}{\mu}$. Then by~\eqref{eq:poincare} and the above inequality
	\begin{align*}
		& \frac{1}{2} \int_\Omega |u(t)|^2 - \e^{-t/\tau} \; \frac{1}{2}\int_\Omega |u_0|^2
			= \int_0^t \frac{\dx}{\dx s} \Bigl( e^{(s-t)/\tau} \; \frac{1}{2} \int_\Omega |u(s)|^2 \Bigr) \\
			& \quad \le \frac{1}{2\tau} \int_0^t \e^{(s-t)/\tau} \int_\Omega |u(s)|^2 \\
				& \qquad + \int_0^t \e^{(s-t)/\tau} \Bigl( c_2 \int_\Omega |f(s)|^2 + c_2 \int_{\partial\Omega} |g(s)|^2 - \frac{\mu}{2} \int_\Omega |\nabla u(s)|^2 \Bigr) \\
			& \quad \le c_2 \int_0^t \e^{(s-t)/\tau} \Bigl( \int_\Omega |f(s)|^2 + \int_{\partial\Omega} |g(s)|^2 \Bigr),
	\end{align*}
	where in the last step we have used that $\frac{c_2}{2\tau} = \frac{\mu}{2}$.
\end{proof}

We want to find a condition on $f$ and $g$ which ensures that the right hand side of~\eqref{eq:expdecayest}
remains bounded as $t \to \infty$. To this end we introduce some function spaces.
\begin{definition}\label{def:Rspace}
	Let $r_1$ and $q_1$ be in $[1,\infty)$, and let $T > 0$.
	For a strongly measurable function $f\colon (0,\infty) \to L^{q_1}(\Omega)$ we define
	\[
		R_{f,T}^{r_1,q_1}(t) \coloneqq \bigl\|f|_{(t,t+T)}\bigr\|_{L^{r_1}(t,t+T; L^{q_1}(\Omega))}
			= \Bigl( \int_0^\infty \|f(s)\|_{L^{q_1}(\Omega)}^{r_1} \setone_{(t,t+T)}(s) \Bigr)^{\frac{1}{r_1}}
	\]
	and introduce the spaces
	\[
		L^{r_1,q_1}_m(\Omega) \coloneqq \bigl\{ f\colon (0,\infty) \to L^{q_1}(\Omega) \mid R_{f,T}^{r_1,q_1} \in L^\infty(0,\infty) \bigr\}
	\]
	and
	\[
		L^{r_1,q_1}_{m,0}(\Omega) \coloneqq \bigl\{ f \in L^{r_1,q_1}_m \mid \lim_{t \to \infty} R_{f,T}^{r_1,q_1}(t) = 0 \bigr\}
	\]
	of uniformly mean integrable functions,
	where we identify functions that coincide almost everywhere.
	Similarly, for $r_2$ and $q_2$ in $[1,\infty)$ and $g\colon (0,\infty) \to L^{q_2}(\partial\Omega)$ we set
	\begin{align*}
		R_{g,T}^{r_1,q_1}(t) & \coloneqq \bigl\|g|_{(t,t+T)}\bigr\|_{L^{r_2}(t,t+T; L^{q_2}(\partial\Omega))}, \\
		L^{r_2,q_2}_m(\partial\Omega) & \coloneqq \bigl\{ g\colon (0,\infty) \to L^{q_2}(\partial\Omega) \mid R_{g,T}^{r_2,q_2} \in L^\infty(0,\infty) \bigr\}, \\
		L^{r_2,q_2}_{m,0}(\partial\Omega) & \coloneqq \bigl\{ g \in L^{r_2,q_2}_m \mid \lim_{t \to \infty} R_{g,T}^{r_2,q_2}(t) = 0 \bigr\}.
	\end{align*}
\end{definition}

Let us collect a few properties of the spaces introduced in Definition~\ref{def:Rspace}.
\begin{lemma}\label{lem:Lpm}
	Let $r_1$ and $q_1$ be in $[1,\infty)$. Then
	\begin{enumerate}[(a)]
	\item\label{it:Lpm_norm}
		for every $T > 0$, the expression $\|f\|_{L^{r_1,q_1}_m(\Omega)} \coloneqq \sup_{t \ge 0} R_{f,T}^{r_1,q_1}(t)$
		defines a complete norm on $L^{r_1,q_1}_m(\Omega)$;
	\item\label{it:Lpm_equiv}
		the norms in~\eqref{it:Lpm_equiv} are pairwise equivalent for different values of $T$;
	\item\label{it:Lpm_cont}
		for every $f \in L^{r_1,q_1}_m(\Omega)$ and every $T > 0$ the function $R_{f,T}^{r_1,q_1}$ is continuous
		on $[0,\infty)$;
	\item\label{it:Lpm_subspace}
		the space $L^{r_1,q_1}_{m,0}(\Omega)$ is a closed subspace of $L^{r_1,q_1}_m(\Omega)$;
	\item\label{it:Lpm_embedding}
		if $1 \le r_1' \le r_1$ and $1 \le q_1' \le q_1$, then
		\[
			L^{r_1,q_1}_m(\Omega) \subset L^{r_1',q_1'}_m(\Omega)
			\quad\text{and}\quad
			L^{r_1,q_1}_{m,0}(\Omega) \subset L^{r_1',q_1'}_{m,0}(\Omega)
		\]
		with continuous embeddings;
	\item\label{it:Linf_contained}
		we have $L^\infty(0,\infty; L^{q_1}(\Omega)) \subset L^{r_1,q_1}_m(\Omega)$
		and $\mathrm{C}_0([0,\infty); L^{q_1}(\Omega)) \subset L^{r_1,q_1}_{m,0}(\Omega)$
		with continuous embeddings;
	\item\label{it:Lpm_convest}
		for $f \in L^{r_1,q_1}_m(\Omega)$ and every non-increasing function
		$h \in L^1(0,\infty) \cap L^\infty(0,\infty)$ we have
		\[
			\int_0^t h(t-s) \; \|f(s)\|_{L^{q_1}(\Omega)}^{r_1} \; \dx s
				\le \bigl( \|h\|_{L^\infty(0,\infty)} + \tfrac{2}{T} \|h\|_{L^1(0,\infty)} \bigr) \, \|R_{f,T}^{r_1,q_1}\|_{L^\infty(0,\infty)}^{r_1}
		\]
		for all $T > 0$ and $t \ge 0$;
	\item\label{it:Lpm_conv}
		for $f \in L^{r_1,q_1}_{m,0}(\Omega)$ and every non-increasing function
		$h \in L^1(0,\infty) \cap L^\infty(0,\infty)$ we have
		\[
			\lim_{t \to 0} \int_0^t h(t-s) \; \|f(s)\|_{L^{q_1}(\Omega)}^{r_1} \; \dx s = 0.
		\]
	\end{enumerate}
	Analogous assertions hold for the spaces $L^{r_2,q_2}_m(\partial\Omega)$ and $L^{r_2,q_2}_{m,0}(\partial\Omega)$
	with $r_2,q_2 \in [1,\infty)$.
\end{lemma}
Part~\eqref{it:Lpm_equiv} justifies that we suppress the dependence on $T$ in the notation for $L^{r_1,q_1}_m(\Omega)$
and its norm.
\begin{proof}
	Part~\eqref{it:Lpm_norm} is routinely checked and we leave the verification to the reader.
	
	Now let $T > 0$ and $T' > 0$ be given
	and pick a natural number $n \ge \frac{T'}{T}$. Then by H\"older's inequality
	\begin{align*}
		R_{f,T'}^{r_1,q_1}(t)
			& \le R_{f,nT}^{r_1,q_1}(t)
			= \Bigl( \sum_{k=0}^{n-1} R_{f,T}^{r_1,q_1}(t+kT)^{r_1} \Bigr)^{\frac{1}{r_1}} \\
			& \le \sum_{k=0}^{n-1} R_{f,T}^{r_1,q_1}(t+kT)
			\le n \; \sup_{s \ge 0} R_{f,T}^{r_1,q_1}(s)
	\end{align*}
	for all $t \ge 0$, which implies~\eqref{it:Lpm_equiv}.

	By the reverse triangle inequality we have
	\[
		\bigl| R_{f,T}^{r_1,q_1}(t+h) - R_{f,T}^{r_1,q_1}(t) \bigr|
			\le \Bigl( \int_0^\infty \|f(s)\|_{L^{q_1}(\Omega)}^{r_1} \bigl| \setone_{(t+h,t+T+h)}(s) - \setone_{(t,t+T)}(s) \bigr| \Bigr)^{\frac{1}{r_1}}.
	\]
	Since moreover $\setone_{(t+h,t+T+h)} \to \setone_{(t,t+T)}$ almost everywhere as $h \to 0$,
	part~\eqref{it:Lpm_cont} follows from the dominated convergence theorem, where as dominating
	function we may take $\|f\|_{L^{q_1}(\Omega)}^{r_1} \setone_{(0,t+2T)} \in L^1(0,\infty)$.

	By~\eqref{it:Lpm_cont} and the definition of the norm the mapping $f \mapsto R_{f,T}^{r_1,q_1}$ is Lipschitz continuous
	from $L^{r_1,q_1}_m(\Omega)$ to $\mathrm{C}_b([0,\infty))$ for every $T > 0$. Hence the preimage of $\mathrm{C}_0([0,\infty))$
	under this function is closed, which proves~\eqref{it:Lpm_subspace}.

	For $1 \le r_1' \le r_1$ and $1 \le q_1' \le q_1$ we obtain from H\"older's inequality that
	\[
		R_{f,T}^{r_1',q_1'}(t) \le T^{\frac{r_1 - r_1'}{r_1 r_1'}} |\Omega|^{\frac{q_1 - q_1'}{q_1 q_1'}} R_{f,T}^{r_1,q_1}
	\]
	for all $t \ge 0$. This implies~\eqref{it:Lpm_embedding}, and~\eqref{it:Linf_contained} is proved
	similarly.

	For~\eqref{it:Lpm_convest} let $f \in L^{r_1,q_1}_m(\Omega)$, $t > 0$ and $T > 0$ be fixed and
	define $n_t \in \mathds{N}$ by $(n_t-1)T \le t < n_tT$.
	Let $h \in L^1(0,\infty) \cap L^\infty(0,\infty)$ be non-increasing and assume
	without loss of generality that $h(0) = \|h\|_{L^\infty(0,\infty)}$.
	Since for $t \le T$ the estimate in~\eqref{it:Lpm_convest} is trivial, we may assume that $t \ge T$, i.e., $n_t \ge 2$.
	Then
	\begin{equation}\label{eq:sumhest}
		\sum_{k=0}^{n_t-1} h\bigl(\tfrac{(n_t-k)t}{n_t}\bigr)
			\le \frac{n_t}{t} \sum_{k=0}^{n_t-1} \int_{\frac{(n_t-k-1)t}{n_t}}^{\frac{(n_t-k)t}{n_t}} h(s)
			\le \frac{2}{T} \int_0^t h(s).
	\end{equation}
	Moreover,
	\begin{equation}\label{eq:convfirst}
		\begin{aligned}
			\int_0^t h(t-s) \; \|f(s)\|_{L^{q_1}(\Omega)}^{r_1} \; \dx s
				& \le \sum_{k=1}^{n_t} h(t - \tfrac{k}{n_t}t) \int_{(k-1)\frac{t}{n_t}}^{k\frac{t}{n_t}} \|f(s)\|_{L^{q_1}(\Omega)}^{r_1} \dx s \\
				& \le \sum_{k=1}^{n_t} h\bigl(\tfrac{(n_t-k)t}{n_t}\bigr) \bigl( R_{f,T}^{r_1,q_1}\bigl(\tfrac{(k-1)t}{n_t}\bigr) \bigr)^{r_1}.
		\end{aligned}
	\end{equation}
	The estimate in~\eqref{it:Lpm_convest} is an immediate consequence of~\eqref{eq:sumhest} and~\eqref{eq:convfirst}.

	Now assume in addition that $f \in L^{r_1,q_1}_{m,0}(\Omega)$. Let $\eps > 0$ be given and pick $k_1 \in \mathds{N}$ so large that
	$R_{f,T}^{r_1,q_1}(s)^{r_1} \le \eps$ for all $s \ge k_1 T$. Let $k_2 \in \mathds{N}$ be so large that
	$h(s) \le \frac{\eps}{2k_1}$ for all $s \ge k_2T$, set $k_0 \coloneqq \max\{4k_1,2k_2\}$ and define $t_0 \coloneqq k_0T$.
	Let $t \ge t_0$ be fixed, so $n_t \ge k_0$. Then for $k \le 2k_1$ we have
	\[
		\frac{(n_t-k)t}{n_t} = \Bigl(1 - \frac{k}{n_t}\Bigr)t \ge \Bigl( 1 - \frac{2k_1}{k_0} \Bigr) t \ge \frac{t}{2} \ge k_2T,
	\]
	whereas for $k \ge 2k_1+1$ we have
	\[
		\frac{(k-1)t}{n_t} \ge \frac{2k_1t}{2(n_t-1)} \ge k_1T.
	\]
	Hence from~\eqref{eq:sumhest} and the definitions of $k_1$ and $k_2$ we obtain for $t \ge k_0T$ that
	\begin{align*}
		& \sum_{k=1}^{n_t} h\bigl(\tfrac{(n_t-k)t}{n_t}\bigr) \bigl( R_{f,T}^{r_1,q_1}\bigl(\tfrac{(k-1)t}{n_t}\bigr) \bigr)^{r_1} \\
			& \qquad \le \frac{\eps}{2k_1} \sum_{k=1}^{2k_1} \bigl( R_{f,T}^{r_1,q_1}\bigl(\tfrac{(k-1)t}{n_t}\bigr) \bigr)^{r_1}
				+ \eps \sum_{k=2k_1+1}^{n_t} h\bigl(\tfrac{(n_t-k)t}{n_t}\bigr) \\
			& \qquad \le \eps \Bigl( \|R_{f,T}^{r_1,q_1}\|_{L^\infty(0,\infty)}^{r_1} + h(0) + \|h\|_{L^1(0,\infty)} \Bigr).
	\end{align*}
	We have shown that
	\[
		\lim_{t \to 0} \sum_{k=1}^{n_t} h\bigl(\tfrac{(n_t-k)t}{n_t}\bigr) \bigl( R_{f,T}^{r_1,q_1}\bigl(\tfrac{(k-1)t}{n_t}\bigr) \bigr)^{r_1} = 0,
	\]
	which by~\eqref{eq:convfirst} implies~\eqref{it:Lpm_conv}.
\end{proof}

We can now formulate our criterion for boundedness and convergence of solutions of $(P_{u_0,f,g})$, which together with
its corollary is the main result of this section.
\begin{theorem}\label{thm:asL2}
	If~\eqref{eq:assc} and~\eqref{eq:assb} hold, then for all $u_0 \in L^2(\Omega)$,
	$f \in L^{2,2}_m(\Omega)$ and $g \in L^{2,2}_m(\partial\Omega)$ that satisfy~\eqref{eq:fgintz}
	the weak solution $u$ of $(P_{u_0,f,g})$ is bounded in $L^2(\Omega)$, and more precisely
	\[
		\int_\Omega |u(t)|^2 \le c \int_\Omega |u_0|^2 + c \, \|f\|_{L^{2,2}_m(\Omega)}^2 + c \, \|g\|_{L^{2,2}_m(\partial\Omega)}^2
	\]
	for all $t \ge 0$ with a constant $c \ge 0$ that depends only on $\Omega$ and the coefficients.
	If even $f \in L^{2,2}_{m,0}(\Omega)$ and $g \in L^{2,2}_{m,0}(\partial\Omega)$, then
	$\lim_{t \to \infty} u(t) = \frac{1}{|\Omega|} \int_\Omega u_0$ in $L^2(\Omega)$.
\end{theorem}
\begin{proof}
	Write $u_0 = \hat{u}_0 + k$ with $k \coloneqq \frac{1}{|\Omega|} \int_\Omega u_0$.
	Then $u(t) = \hat{u}(t) + k$ by Lemma~\ref{lem:Linfcontr}, where $\hat{u}$ denotes the weak solution
	of~$(P_{\hat{u}_0,f,g})$.
	Proposition~\ref{prop:L2bound} and part~\eqref{it:Lpm_convest} of Lemma~\ref{lem:Lpm}
	applied with $h(r) \coloneqq \e^{-r/\tau}$ show that
	\[
		\int_\Omega |\hat{u}(t)|^2 \le c \int_\Omega |\hat{u}_0|^2 + c \, \|f\|_{L^{2,2}_m(\Omega)}^2 + c \, \|g\|_{L^{2,2}_m(\partial\Omega)}^2,
	\]
	whereas part~\eqref{it:Lpm_conv} shows that $\lim_{t \to \infty} \hat{u}(t) = 0$ in $L^2(\Omega)$ if $f \in L^{2,2}_{m,0}(\Omega)$ and $g \in L^{2,2}_{m,0}(\partial\Omega)$.
\end{proof}

Under slightly stronger assumptions on $u_0$, $f$ and $g$ we obtain even uniform boundedness and uniform convergence.
\begin{corollary}\label{cor:asLinf}
	Let $r_1$, $q_1$, $r_2$ and $q_2$ be numbers in $[2,\infty)$ that satisfy~\eqref{eq:qrrel}.
	If~\eqref{eq:assc} and~\eqref{eq:assb} hold, then for all $u_0 \in L^\infty(\Omega)$,
	$f \in L^{r_1,q_1}_m(\Omega)$ and $g \in L^{r_2,q_2}_m(\partial\Omega)$ which satisfy~\eqref{eq:fgintz}
	the weak solution $u$ of $(P_{u_0,f,g})$ is bounded in $L^\infty(\Omega)$, and more precisely
	\begin{equation}\label{eq:globLinfbound}
		\|u(t)\|_{L^\infty(\Omega)}^2 \le c \, \|u_0\|_{L^\infty(\Omega)}^2 + c \, \|f\|_{L^{r_1,q_1}_m(\Omega)}^2 + c \, \|g\|_{L^{r_2,q_2}_m(\partial\Omega)}^2
	\end{equation}
	for all $t \ge 0$.
	If even $f \in L^{r_1,q_1}_{m,0}(\Omega)$ and $g \in L^{r_2,q_2}_{m,0}(\partial\Omega)$, then
	$\lim_{t \to \infty} u(t) = \frac{1}{|\Omega|} \int_\Omega u_0$ in $L^\infty(\Omega)$.
\end{corollary}
\begin{proof}
	By Theorem~\ref{thm:asL2} and part~\eqref{it:Lpm_embedding} of Lemma~\ref{lem:Lpm} we have
	\[
		\|u\|_{L^2(\Omega)}^2 \le c \, \|u_0\|_{L^\infty(\Omega)}^2 + c \, \|f\|_{L^{r_1,q_1}_m(\Omega)}^2 + c \, \|g\|_{L^{r_2,q_2}_m(\partial\Omega)}^2.
	\]
	On the other hand, inequality~\eqref{eq:Linfbound} applied to the interval $[t-2,t]$ shows that
	\begin{equation}\label{eq:Linfhelperest}
		\|u(t)\|_{L^\infty(\Omega)}^2
			\le 2c \sup_{s \ge t-2} \|u(s)\|_{L^2(\Omega)}^2 + c \, \bigl( R_{f,2}^{r_1,q_1}(t-2) \bigr)^2 + c \, \bigl( R_{g,2}^{r_1,q_1}(t-2) \bigr)^2
	\end{equation}
	for every $t \ge 2$.
	Using in addition Theorem~\ref{thm:inLinf} to bound $u$ on $[0,2]$, we have shown~\eqref{eq:globLinfbound}.

	Let now $f \in L^{r_1,q_1}_{m,0}(\Omega) \subset L^{2,2}_{m,0}(\Omega)$ and
	$g \in L^{r_2,q_2}_{m,0}(\partial\Omega) \subset L^{2,2}_{m,0}(\partial\Omega)$, see Lemma~\ref{lem:Lpm}.
	Write $u(t) = \hat{u}(t) + k$ with $k \coloneqq \frac{1}{|\Omega|} \int_\Omega u_0$ as in the proof of
	Theorem~\ref{thm:asL2}. Then $\lim_{t \to \infty} \|\hat{u}(t)\|_{L^2(\Omega)} = 0$ by Theorem~\ref{thm:asL2}.
	Using the definitions of $L^{r_1,q_1}_{m,0}(\Omega)$ and $L^{r_1,q_1}_{m,0}(\partial\Omega)$,
	this gives $\lim_{t \to \infty} \|\hat{u}(t)\|_{L^\infty(\Omega)} = 0$ by~\eqref{eq:Linfhelperest} applied to $\hat{u}$.
	The additional claim is proved.
\end{proof}

\begin{remark}
	Remark~\ref{rem:latercont} shows that
	if in the situation of Corollary~\ref{cor:asLinf} we only have $u_0 \in L^2(\Omega)$ instead of $u_0 \in L^\infty(\Omega)$,
	the assertions remain valid apart with the exception that $u$ will not be bounded in $L^\infty(\Omega)$ as $t \to 0$, i.e., estimate~\eqref{eq:globLinfbound}
	holds only for $t \ge t_0 > 0$ with a constant $c \ge 0$ that depends in addition on $t_0$.
\end{remark}

\section{Periodicity}\label{sec:asymp}

We are going to study the periodic behavior of solutions of $(P_{u_0,f,g})$ under periodicity assumptions on $f$ and $g$.
This relies on spectral theory, which is why
in this section (and only in this section) we assume our Banach spaces to be complex.
Thus $u_0$, $f$ and $g$ are complex-valued functions, and hence also the solution $u$
will be complex-valued. For the theory developed in the other sections
this makes no difference since we can always treat the real and the imaginary part separately
as long as the coefficients of the equation are real-valued, which we still assume.
Thus we will neglect this detail in the notation and reuse the symbols for the real spaces
for their complex counterparts.

We start this section with a short summary on almost periodic functions in the sense of Harald Bohr, i.e.,
uniformly almost periodic functions.
For further details and proofs we refer to~\cite[\S 4.5--4.7]{ABHN01} or~\cite{Bohr47}.

\begin{definition}
	Let $X$ be a complex Banach space.
	A function $f\colon (0,\infty) \to X$ is called $\tau$-periodic (for some $\tau > 0$)
	if $f(t+\tau) = f(t)$ for all $t \ge 0$. Set $\e_{i \eta}(t) \coloneqq \e^{i\eta t}$
	for $\eta \in \mathds{R}$ and $t \ge 0$.
	The members of the space
	\[
		\AP([0,\infty); X) \coloneqq \overline{\hull} \bigl\{ \e_{i\eta} x : \eta \in \mathds{R}, \; x \in X \bigr\},
	\]
	are called \emph{uniformly almost periodic functions},
	where the closure is taken in the space of bounded, uniformly continuous functions $\BUC([0,\infty); X)$, which
	is a Banach space for the uniform norm.
	The direct topological sum
	\[
		\AAP([0,\infty);X) \coloneqq \AP([0,\infty);X) \oplus \mathrm{C}_0( [0,\infty); X) \subset \BUC( [0,\infty); X)
	\]
	is called the space of \emph{uniformly asymptotically almost periodic functions}.
	For all $f \in \AAP( [0,\infty); X )$ and $\eta \in \mathds{R}$ the Ces\`aro limit
	\[
		C_\eta f \coloneqq \lim_{T \to \infty} \frac{1}{T} \int_0^T \e^{-i \eta s} f(s) \, \dx s
	\]
	exist in $X$. We let
	\[
		\Freq(f) \coloneqq \bigl\{ \eta \in \mathds{R} : C_\eta f \neq 0 \bigr\}
	\]
	denote the \emph{set of frequencies of $f$}.
	For $f \in \AAP([0,\infty);X)$ the set $\Freq(f)$ is countable.
	The function $f$ can be decomposed into its frequencies in the sense that
	\[
		f \in \overline{\hull} \bigl\{ \e_{i\eta} x : \eta \in \Freq(f), \; x \in X \bigr\} \oplus \mathrm{C}_0([0,\infty);X).
	\]
	In particular, $f \in \mathrm{C}_0([0,\infty);X)$ if and only if $\Freq(f) = \emptyset$.
	Moreover, $\Freq(f) \subset \frac{2\pi}{\tau} \mathds{Z}$ if and only there exists a $\tau$-periodic
	function $g$ such that $f-g \in \mathrm{C}_0([0,\infty);X)$.
\end{definition}

We show that for uniformly asymptotically almost periodic data, the solution is uniformly asymptotically
almost periodic with essentially the same frequencies. In fact, this is a general phenomenon for
mild solutions of abstract Cauchy problems and we merely have to check the assumptions of~\cite[Corollary~5.6.9]{ABHN01}.
We are going to improve this result later, which is why we call this preliminary result a lemma.

\begin{lemma}\label{lem:perL2}
	Assume~\eqref{eq:assc} and~\eqref{eq:assb}
	and let $u_0 \in L^2(\Omega)$, $f \in \AAP([0,\infty); L^2(\Omega))$ and $g \in \AAP([0,\infty); L^2(\partial\Omega))$
	satisfy~\eqref{eq:fgintz}. Then the weak solution $u$ of $(P_{u_0,f,g})$ is in $\AAP([0,\infty); L^2(\Omega))$.
\end{lemma}
\begin{proof}
	Define $u_h(t) \coloneqq u(t+h)$, $f_h(t) \coloneqq f(t+h)$
	and $g_h(t) \coloneqq g(t+h)$ for $h \ge 0$ and $t \ge 0$. Then by uniform continuity of $f$
	for every $\eps > 0$ there exists $\delta > 0$ such that
	$\|f_h - f\|_{L^{2,2}_m(\Omega)} \le \eps$ holds
	whenever $0 \le h < \delta$, see part~\eqref{it:Linf_contained} of Lemma~\ref{lem:Lpm}.
	A similar assertion holds for $g$.
	Applying Theorem~\ref{thm:asL2} to $u$ and $u_h - u$, which is the weak solution of
	$(P_{u(h) - u(0), f_h - f, g_h - g})$, and using in addition that $u$ is continuous
	by Definition~\ref{def:weaksol} we thus obtain that $u \in \BUC([0,\infty); L^2(\Omega))$.

	Let $A_2$ be as in Definition~\ref{def:weakderiv}. By Lemma~\ref{lem:DA2} the operator $A_2$
	generates a once integrated semigroup $(S(t))_{t \ge 0}$ on $L^2(\Omega) \times L^2(\partial\Omega)$,
	see~\cite[Theorem~3.11.7]{ABHN01}, which by~\cite[Lemma~3.2.9]{ABHN01} satisfies
	$S(t)(v,0) = (\int_0^t T_{2,h}(s) \, v, 0)$ for all $v \in L^2(\Omega)$,
	where $(T_{2,h}(t))_{t \ge 0}$ is defined in Proposition~\ref{prop:semigroup}. By
	Proposition~\ref{prop:charpres} the closed subspace
	\[
		X_0 \coloneqq \Bigl\{ (v,0) : v \in L^2(\Omega), \; \int_\Omega v = 0 \Bigr\}
	\]
	of $L^2(\Omega) \times L^2(\partial\Omega)$
	is invariant under the action of $(S(t))_{t \ge 0}$, which by~\cite[Definition~3.2.1]{ABHN01} implies
	that $X_0$ is invariant under the resolvent of $A_2$. Hence for the part $A_2|_{X_0}$ of $A_2$ in $X_0$
	we have $\sigma(A_2|_{X_0}) \subset \sigma(A_2)$ and in particular $\rho(A_2|_{X_0}) \neq \emptyset$.
	We obtain from Lemma~\ref{lem:DA2} and the compactness of the embedding $H^1(\Omega) \hookrightarrow L^2(\Omega)$
	that $A_2|_{X_0}$ has compact resolvent.

	We now show that $\sigma(A_2|_{X_0}) \cap i\mathds{R} = \emptyset$.
	Assume to the contrary that there exists $\eta \in \mathds{R}$ such that
	$i\eta \in \sigma_p(A_2|_{X_0}) = \sigma(A_2|_{X_0})$.
	Then there exists $0 \neq v_0 \in L^2(\Omega)$ satisfying $\int_\Omega v_0 = 0$ and $A_2(v_0,0) = (i\eta \, v_0, 0)$.
	Then $v(t) \coloneqq \e^{i\eta t} v_0$ defines a classical $L^2$-solution of $(P_{v_0,0,0})$.
	This contradicts Proposition~\ref{prop:L2bound} because $\|v(t)\|_{L^2(\Omega)}^2 \not\to 0$ as $t \to \infty$.

	Write $u_0 = \hat{u}_0 + k$ with $k \coloneqq \frac{1}{|\Omega|} \int_\Omega u_0$.
	Then $u(t) = \hat{u}(t) + k$ by Lemma~\ref{lem:Linfcontr}, where $\hat{u}$ is the weak (and hence mild) solution
	of $(P_{\hat{u}_0, f, g})$.
	Since in addition $\int_\Omega u(t) = 0$ for all $t \ge 0$ by Proposition~\ref{prop:charpres},
	we deduce that $(u,0)$ is a mild solution of the abstract Cauchy problem associated with $A_2|_{X_0}$
	for the inhomogeneity $(f,g)$.
	Since $\hat{u} \in \BUC([0,\infty); L^2(\Omega))$
	we now obtain from~\cite[Corollary~5.6.9]{ABHN01} that $\hat{u} \in \AAP([0,\infty); L^2(\Omega))$,
	which shows $u \in \AAP([0,\infty); L^2(\Omega))$.
\end{proof}

Via an approximation argument we can relax the assumptions of Lemma~\ref{lem:perL2}.
For this we introduce Stepanoff almost periodic functions. We omit the proofs of the implicit statements
about this class of functions, which are similar to the ones for uniformly almost periodic functions.
The interested reader may consult~\cite[\S 99]{Bohr47} and~\cite{Ste26} for the scalar-valued case.
\begin{definition}
	Let $X$ be a complex Banach space.
	For $r \in [1,\infty)$ the members of the space
	\[
		\AP^r([0,\infty); X) \coloneqq \overline{\hull} \bigl\{ \e_{i\eta} x : \eta \in \mathds{R}, \; x \in X \bigr\},
	\]
	are called \emph{Stepanoff almost periodic functions (to the exponent $r$)},
	where the closure is taken with respect to the norm
	\[
		\|f\|_{L^r_m(X)} \coloneqq \sup_{t \ge 0} \Bigl( \int_t^{t+1} \|f(s)\|_X^r \Bigr)^{\frac{1}{r}}.
	\]
	The space of \emph{Stepanoff asymptotically almost periodic functions} is defined as
	\[
		\AAP^r([0,\infty);X) \coloneqq \AP^r([0,\infty);X) \oplus L^r_{m,0}(X),
	\]
	where we set $L^r_{m,0}(X) \coloneqq \bigl\{ f \in L^r_m(X) : \lim_{t \to \infty} \int_t^{t+1} \|f(s)\|^r \to 0 \bigr\}$.
	The Ces\`aro limit
	\[
		C_\eta \coloneqq \lim_{T \to \infty} \frac{1}{T} \int_0^T f(s)
	\]
	exists for all $\eta \in \mathds{R}$ and $f \in \AAP^r([0,\infty); X)$. We define the set of frequencies of $f$ as
	\[
		\Freq(f) \coloneqq \bigl\{ \eta \in \mathds{R} : C_\eta f \neq 0 \bigr\}
	\]
	and remark that $\Freq(f) \subset \frac{2\pi}{\tau} \mathds{Z}$ if and only there exists a $\tau$-periodic
	function $g$ such that $f-g \in L^r_{m,0}(X)$.
\end{definition}

Now improve the statement of Lemma~\ref{lem:perL2} by showing that for Stepanoff asymptotically almost periodic data
we obtain uniformly asymptotically almost periodic solutions with a precise description of their frequencies.
We start with the result in the $L^2$-framework.
\begin{theorem}\label{thm:perL2}
	Assume that~\eqref{eq:assc} and~\eqref{eq:assb} hold. We assume that
	$u_0 \in L^2(\Omega)$, $f \in \AAP^2([0,\infty); L^2(\Omega))$ and $g \in \AAP^2([0,\infty); L^2(\partial\Omega))$
	satisfy~\eqref{eq:fgintz}. Then the weak solution $u$ of $(P_{u_0,f,g})$ is in $\AAP([0,\infty); L^2(\Omega))$.
	For $\eta \neq 0$ we have $\eta \in \Freq(u)$ if and only if $\eta \in \Freq(f) \cup \Freq(g)$.
	Moreover, $0 \in \Freq(u)$ if and only if $0 \in \Freq(f) \cup \Freq(g)$ or $\int_\Omega u_0 \neq 0$.
\end{theorem}
\begin{proof}
	Write $f = f_P + f_C$ with $f_P \in \AP([0,\infty); L^2(\Omega))$ and $f_C \in L^2_{m,0}(L^2(\Omega))$,
	$g = g_P + g_C$ with $g_P \in \AP([0,\infty); L^2(\partial\Omega))$ and $g_C \in L^2_{m,0}(L^2(\partial\Omega))$
	and $u_0 = \hat{u}_0 + k$ with $k \coloneqq \frac{1}{|\Omega|} \int_\Omega u_0$.
	Then $u = u_P + u_C + k$ by Lemma~\ref{lem:Linfcontr}, where $u_P$ denotes the solution of $(P_{\hat{u}_0,f_P,g_P})$
	and $u_C$ is the solution of $(P_{0,f_C,g_C})$.

	Pick
	$f_n \in \hull\{ \e_{i\eta} v : \eta \in \mathds{R}, \; v \in L^2(\Omega) \}$
	and
	$g_n \in \hull\{ \e_{i\eta} w : \eta \in \mathds{R}, \; w \in L^2(\partial\Omega) \}$
	such that $f_n \to f$ in the norm of $L^2_m(L^2(\Omega)) = L^{2,2}_m(\Omega)$ and
	$g_n \to g$ in the norm of $L^2_m(L^2(\partial\Omega)) = L^{2,2}_m(\partial\Omega)$.
	Let $u_n$ denote the weak solution of $(P_{\hat{u}_0,f_n,g_n})$.
	Then $u_n \to u_P$ in $L^\infty(0,\infty; L^2(\Omega))$ by Theorem~\ref{thm:asL2} and
	$u_n \in \AAP([0,\infty); L^2(\Omega))$ by Lemma~\ref{lem:perL2}. Hence
	$u_P \in \AAP([0,\infty); L^2(\Omega))$.
	Since $(u_n,0)$ is a mild solution of the abstract Cauchy problem associated with $A_2|_{X_0}$ for the
	inhomogeneity $(f_n,g_n)$, see the proof of Lemma~\ref{lem:perL2} we obtain from~\cite[Proposition~5.6.7]{ABHN01}
	that $C_\eta u_n = (i\eta - A_2|_{X_0})^{-1} (C_\eta f_n, C_\eta g)$ for all $\eta \in \mathds{R}$.
	Passing to the limit we have the relation $C_\eta u_P = (i\eta - A_2|_{X_0})^{-1} (C_\eta f, C_\eta g)$. Thus $\Freq(u_P) = \Freq(f) \cup \Freq(g)$.

	Since $u_C \in \mathrm{C}_0([0,\infty); L^2(\Omega))$ by Theorem~\ref{thm:asL2} and $u_P(t) \perp k$
	for all $t \ge 0$ by Proposition~\ref{prop:charpres}, we deduce that $u \in \AAP([0,\infty); L^2(\Omega))$ and
	\[
		\Freq(u) = \Freq(u_P) + \Freq(k) = \Freq(f) \cup \Freq(g) \cup \Freq(k),
	\]
	which is a different way to write down the description of $\Freq(u)$.
\end{proof}

We can also obtain an analogue of Theorem~\ref{thm:perL2} in the more regular setting of continuous solutions.
\begin{theorem}\label{thm:perLinf}
	Let $r_1$, $q_1$, $r_2$ and $q_2$ be numbers in $[2,\infty)$ that satisfy relation~\eqref{eq:qrrel}.
	Assume that~\eqref{eq:assc} and~\eqref{eq:assb} hold
	and let $u_0 \in L^\infty(\Omega)$, $f \in \AAP^{r_1}([0,\infty); L^{q_1}(\Omega))$ and $g \in \AAP^{r_2}([0,\infty); L^{q_2}(\partial\Omega))$
	satisfy~\eqref{eq:fgintz}. Then the weak solution $u$ of $(P_{u_0,f,g})$ is in $\AAP([0,\infty); L^\infty(\Omega))$.
	For $\eta \neq 0$ we have $\eta \in \Freq(u)$ if and only if $\eta \in \Freq(f) \cup \Freq(g)$.
	Moreover, $0 \in \Freq(u)$ if and only if $0 \in \Freq(f) \cup \Freq(g)$ or $\int_\Omega u_0 \neq 0$.
	If $u_0 \in \mathrm{C}(\overline{\Omega})$, then $u \in \AAP([0,\infty); \mathrm{C}(\overline{\Omega}))$.
\end{theorem}
\begin{proof}
	This theorem can be proved in precisely the same way as Theorem~\ref{thm:perL2}. We have to use Corollary~\ref{cor:asLinf} instead of Theorem~\ref{thm:asL2}
	and the realization of $A$ in $L^{q_1}(\Omega) \times L^{q_2}(\partial\Omega)$ instead of $A_2$ like in Theorem~\ref{thm:inhomcont}, from
	where we also obtain the continuity of $u$ if $u_0 \in \mathrm{C}(\overline{\Omega})$.
	We leave the details to the reader.
\end{proof}

As an immediate consequence of the previous two theorems, we see that for periodic data the solution is asymptotically periodic.
This formulation is simpler, but we lose the precise information about the frequencies.
\begin{corollary}
	Assume that~\eqref{eq:assc} and~\eqref{eq:assb} hold. Fix functions
	$u_0 \in L^2(\Omega)$, $f \in L^2(0,\tau; L^2(\Omega))$ and $g \in L^2(0,\tau; L^2(\partial\Omega))$ for some $\tau > 0$.
	We identify $f$ and $g$ with their $\tau$-periodic extensions to $(0,\infty)$. Then there exists a $\tau$-periodic
	function $u_P$ such that the weak solution $u$ of $(P_{u_0,f,g})$ satisfies $\lim_{t \to \infty} \|u(t) - u_P(t)\|_{L^2(\Omega)} = 0$.
	If $u_0 \in \mathrm{C}(\overline{\Omega})$, $f \in L^\infty(0,\tau; L^\infty(\Omega))$ and $g \in L^\infty(0,\tau; L^\infty(\partial\Omega))$,
	then $u$ and $u_P$ are in $\mathrm{C}_b([0,\infty); \mathrm{C}(\overline{\Omega}))$ and $\lim_{t \to \infty} \|u(t) - u_P(t)\|_{L^\infty(\Omega)} = 0$.
\end{corollary}

\appendix
\section{Pointwise estimates via De Giorgi's techniques}\label{app:pointest}

In this section we prove Proposition~\ref{prop:Linfbound}. The proof is similar to what can be found in~\cite[\S III.7--8]{LSU67}, which 
in term is a refined version of De Giorgi's famous technique. We need, however, the following improvements over~\cite{LSU67}:
\begin{enumerate}[(i)]
\item
	the presence of the inhomogeneity $g$ in $(P_{u_0,f,g})$,
	makes it necessary to keep track of the measure of the sublevel sets of $u|_{\partial\Omega}$;
\item
	we need a precise dependence of the constants on $f$ and $g$. More precisely, these quantities
	have to enter linearly into the right hand side. This is not obvious from the proofs
	in~\cite{LSU67}, but can be asserted after some small modifications;
\item
	we need an estimate that is local in time but global in space, whereas the results in~\cite{LSU67}
	are either global in both variables or local. This requires only trivial modifications.
\end{enumerate}
Another motivation to give the details is that relevant parts in~\cite{LSU67} contain some misprints.
For example, the relations between $n$, $\hat{r}$
and $\hat{q}$ in the proof of~\cite[Theorem~III.7.1]{LSU67} are faulty, as can be seen by taking
$n = 2$, $r = q = 4$ and $\kappa = 1/2$.

A more subtle mistake is the claim that the constant in~\cite[(II.6.11)]{LSU67} does not depend on $\tau_0$ and $\rho_0$.
This is wrong, which renders the seemingly precise elaboration of the dependence on $\tau_0$
and $\rho_0$ useless. More precisely, a closer look at the proof exhibits that the
explicit constant given in~\cite[(II.6.25)]{LSU67} still contains $\theta = \tau_0 \rho_0^{-2}$.
In fact, if estimate~\cite[(II.6.11)]{LSU67} was true, then applying it to the solution $u$
of the heat equation with initial datum $u_0 \in L^2(\mathds{R}^N) \setminus L^\infty(\mathds{R}^N)$
like in~\cite[\S III.8]{LSU67} we could deduce that given a ball $B \subset \mathds{R}^N$ we have
\[
	\sup_{\frac{T}{2} \le t \le T} \|u(t)\|_{L^\infty(B)}^2 \le c \|u_0\|_{L^2(\mathds{R}^N)}^2
\]
for all $T > 0$ with a constant $c \ge 0$ that depends only on the radius of the ball, which contradicts
that $u(t) \to u_0$ in $L^2(\mathds{R}^N)$.

\bigskip

For these reasons, we give a complete proof of Proposition~\ref{prop:Linfbound}. The only part of the argument
that we copy from~\cite{LSU67} without change is the following lemma, which is easily proved by induction.
\begin{lemma}[{\cite[Lemma~II.5.7]{LSU67}}]\label{lem:verblemma}
	Let $(y_n)_{n \in \mathds{N}_0}$ and $(z_n)_{n \in \mathds{N}_0}$ be sequences of non-negative real numbers such that
	\[
		y_{n+1} \le c b^n \bigl( y_n^{1+\delta} + z_n^{1+\eps} y_n^\delta \bigr)
		\quad\text{and}\quad
		z_{n+1} \le c b^n \bigl( y_n + z_n^{1+\eps} \bigr)
	\]
	for all $n \in \mathds{N}_0$
	with positive constants $c$, $b$, $\eps$ and $\delta$, where $b \ge 1$. Define
	\[
		d \coloneqq \min\bigl\{ \delta, \; \tfrac{\eps}{1+\eps} \bigr\}
		\quad\text{and}\quad
		\lambda \coloneqq \min\bigl\{ (2c)^{-\frac{1}{\delta}} b^{-\frac{1}{\delta d}}, \; (2c)^{-\frac{1+\eps}{\eps}} b^{-\frac{1}{\eps d}} \bigr\}
	\]
	and assume that
	\[
		y_0 \le \lambda
		\quad\text{and}\quad
		z_0 \le \lambda^{\frac{1}{1+\eps}}.
	\]
	Then
	\[
		y_n \le \lambda b^{-\frac{n}{d}}
		\quad\text{and}\quad
		z_n \le \bigl( \lambda b^{-\frac{n}{d}} \bigr)^{\frac{1}{1+\eps}}
	\]
	for all $n \in \mathds{N}_0$.
\end{lemma}

\bigskip

We imitate the notation of~\cite{LSU67} to a certain degree.
More precisely, let $\Omega \subset \mathds{R}^N$ be a bounded Lipschitz domain and $T > 0$.
It will be convenient to work with functions defined for negative times, so we will always assume that
$u \in L^\infty(-T,0;L^2(\Omega)) \cap L^2(-T,0;H^1(\Omega))$. In that case we write
\[
	\|u\|_{Q(\tau)}^2 \coloneqq \sup_{-\tau \le t \le 0} \int_\Omega |u(t)|^2 + \int_{-\tau}^0 \int_\Omega |\nabla u(t)|^2.
\]
and for $k \ge 0$ we define
\[
	u^{(k)}(t) \coloneqq (u(t) - k)^+.
\]
In what follows we will frequently need that for $r_1 \in [2,\infty]$, $q_1 \in [2,\frac{2N}{N-2}]$,
$r_2 \in [2,\infty]$ and $q_2 \in [2,\frac{2(N-1)}{N-2}]$ satisfying
\[
	\frac{1}{r_1} + \frac{N}{2q_1} = \frac{N}{4}
	\quad\text{and}\quad
	\frac{1}{r_2} + \frac{N-1}{2q_2} = \frac{N}{4}
\]
we have
\begin{equation}\label{eq:aniso}
	\|u\|_{L^{r_1}(-\tau,0; L^{q_1}(\Omega))} + \|u\|_{L^{r_2}(-\tau,0; L^{q_2}(\partial\Omega)}
		\le c \|u\|_{Q(\tau)},
\end{equation}
where $c \ge 0$ depends only on $\Omega$, $r_1$, $q_1$, $r_2$ and $q_2$. This anisotropic
Sobolev inequality follows from the multiplicative Sobolev inequalities on $\Omega$, see~\cite[\S II.3]{LSU67}.

We start with a modified version of~\cite[Theoerem~II.6.2]{LSU67}.
\begin{theorem}\label{thm:LSUthm1}
	Let $\Omega \subset \mathds{R}^N$ be a bounded Lipschitz domain, $N \ge 2$.
	Fix $T > 0$ and $u \in L^\infty(-T,0;L^2(\Omega)) \cap L^2(-T,0;H^1(\Omega))$.
	Let $r_{1,\ell} \in [2,\infty)$, $q_{1,\ell} \in [2,\frac{2N}{N-2}]$, $r_{2,\ell} \in [2,\infty)$
	and $q_{2,\ell} \in [2, \frac{2(N-1)}{N-2}]$ satisfy
	\begin{equation}\label{eq:qrcond}
		\frac{1}{r_{1,\ell}} + \frac{N}{2q_{1,\ell}} = \frac{N}{4} \quad (1 \le \ell \le L_1)
		\quad\text{and}\quad
		\frac{1}{r_{2,\ell}} + \frac{N-1}{2q_{2,\ell}} = \frac{N}{4} \quad (1 \le \ell \le L_2).
	\end{equation}
	Assume that there exist $\hat{k} \ge 0$, $\gamma \ge 0$ and numbers $\kappa_{1,\ell} > 0$ and $\kappa_{2,\ell} > 0$ such that
	for all $\tau \in (0,T]$, $\sigma \in (0,\frac{1}{2})$ and $k \ge \hat{k}$ we have
	\begin{equation}\label{eq:ukest}
		\begin{aligned}
			\|u^{(k)}\|_{Q((1-\sigma)\tau)}^2
				& \le \frac{\gamma}{\sigma \tau} \int_{-\tau}^0 \int_\Omega |u^{(k)}(t)|^2
					+ \gamma k^2 \sum_{\ell=1}^{L_1} \Bigl( \int_{-\tau}^0 |A_k(t)|^{\frac{r_{1,\ell}}{q_{1,\ell}}} \Bigr)^{\frac{2(1+\kappa_{1,\ell})}{r_{1,\ell}}} \\
					& \qquad + \gamma k^2 \sum_{\ell=1}^{L_2} \Bigl( \int_{-\tau}^0 |B_k(t)|^{\frac{r_{2,\ell}}{q_{2,\ell}}} \Bigr)^{\frac{2(1+\kappa_{2,\ell})}{r_{2,\ell}}}.
		\end{aligned}
	\end{equation}
	Then
	\begin{equation}\label{eq:LSUest}
		\esssup_{(t,x) \in [-\frac{T}{2},0] \times \Omega} u(t,x) \le c \Bigl( \int_{-T}^0 \int_\Omega |u(t)|^2 + \hat{k}^2 \Bigr)^{\frac{1}{2}},
	\end{equation}
	where the constant $c \ge 0$ is independent of $u$ and $\hat{k}$.
\end{theorem}
\begin{proof}
	In the proof the constants $c$, $c_0$, $c_1$ and $c_2$ never depend on $u$ and $\hat{k}$.
	Moreover, $c$ is a generic constant in the sense that it may change its numeric value between occurrences.

	Since $|A_k(t)| \le |\Omega|$ and $|B_k(t)| \le |\partial\Omega|$
	estimate~\eqref{eq:ukest} remains valid if we replace all the $\kappa_{1,\ell}$ and $\kappa_{2,\ell}$ by
	their least member
	\[
		\kappa \coloneqq \min\{ \kappa_{1,1}, \dots, \kappa_{1,L_1}, \kappa_{1,L_1}, \kappa_{2,1}, \kappa_{2,L_2} \} > 0
	\]
	provided we replace $\gamma$ by a larger constant $\gamma'$ that depends on $\kappa_{1,\ell}$, $\kappa_{2,\ell}$, $r_{1,\ell}$, $q_{1,\ell}$
	$r_{2,\ell}$, $q_{2,\ell}$, $T$, $\gamma$, $|\Omega|$ and $|\partial\Omega|$. Thus we may assume without loss of generality
	that $\kappa_{1,\ell} = \kappa$ for all $1 \le \ell \le L_1$ and $\kappa_{2,\ell} = \kappa$ for all $1 \le \ell \le L_2$.

	Let $M \ge \hat{k}$ be arbitrary and define
	\begin{align*}
		\tau_n & \coloneqq (1 + 2^{-(n+1)}) \tfrac{T}{2} \in [\tfrac{T}{2},T], \\
		k_n & \coloneqq (2 - 2^{-n}) M \ge \hat{k}, \\
		y_n & \coloneqq \frac{1}{M^2} \int_{-\tau_n}^0 \int_\Omega |u^{(k_n)}(t)|^2, \\
		z_n & \coloneqq \sum_{\ell=1}^{L_1} \Bigl( \int_{-\tau_n}^0 |A_{k_n}(t)|^{\frac{r_{1,\ell}}{q_{1,\ell}}} \Bigr)^{\frac{2}{r_{1,\ell}}}
			+ \sum_{\ell=1}^{L_2} \Bigl( \int_{-\tau_n}^0 |B_{k_n}(t)|^{\frac{r_{2,\ell}}{q_{2,\ell}}} \Bigr)^{\frac{2}{r_{2,\ell}}}
	\end{align*}
	for all $n \in \mathds{N}_0$. We prove that the sequences $(y_n)$ and $(z_n)$ satisfy the inequalities in Lemma~\ref{lem:verblemma}.

	To this end, let $n \in \mathds{N}_0$ be fixed. From~\eqref{eq:aniso} and the trivial estimate
	\[
		|u^{(k_n)}(t)|^2 \ge (k_{n+1} - k_n)^2 \; \setone_{A_{k_{n+1}}(t)}
	\]
	we obtain that
	\begin{equation}\label{eq:yn1est}
	\begin{aligned}
		M^2 \, y_{n+1} & \le c \; \Bigl( \int_{-\tau_{n+1}}^0 |A_{k_{n+1}}(t)| \Bigr)^{\frac{2}{N+2}} \|u^{(k_{n+1})}\|_{Q(\tau_{n+1})}^2 \\
			& \le c \; \bigl( (k_{n+1} - k_n)^{-2} M^2 y_n\bigr)^{\frac{2}{N+2}} \|u^{(k_{n+1})}\|_{Q(\tau_{n+1})}^2 \\
			& \le c \; 2^{2(n+1)} y_n^{\frac{2}{N+2}} \|u^{(k_{n+1})}\|_{Q(\tau_{n+1})}^2.
	\end{aligned}
	\end{equation}
	Similarly,
	\begin{equation}\label{eq:zn1est}
	\begin{aligned}
		2^{-2(n+1)} M^2 z_{n+1}
			& = (k_{n+1} - k_n)^2 z_{n+1} \\
			& \le \sum_{\ell=1}^{L_1} \biggl( \int_{-\tau_{n+1}}^0 \Bigl( \int_\Omega |u^{(k_n)}(t)|^{q_{1,\ell}} \Bigr)^{\frac{r_{1,\ell}}{q_{1,\ell}}} \biggr)^{\frac{2}{r_{1,\ell}}} \\
				& \qquad + \sum_{\ell=1}^{L_2} \biggl( \int_{-\tau_{n+1}}^0 \Bigl( \int_{\partial\Omega} |u^{(k_n)}(t)|^{q_{2,\ell}} \Bigr)^{\frac{r_{2,\ell}}{q_{2,\ell}}} \biggr)^{\frac{2}{r_{2,\ell}}} \\
			& \le c \|u^{(k_n)}\|_{Q(\tau_{n+1})}^2
	\end{aligned}
	\end{equation}
	Moreover, from~\eqref{eq:ukest} applied with $\tau = \tau_n$ and $\sigma = 1 - \frac{\tau_{n+1}}{\tau_n} \ge 2^{-(n+3)}$ we get that
	\begin{equation}\label{eq:un1est}
	\begin{aligned}
		\|u^{(k_{n+1})}\|_{Q(\tau_{n+1})}^2
			& \le \|u^{(k_n)}\|_{Q(\tau_{n+1})}^2
			\le \frac{\gamma}{\sigma \tau_n} \, M^2 \, y_n + \gamma k_n^2 z_n^{1+\kappa} \\
			& \le \gamma M^2 2^{n+4} \bigl( T^{-1} + 1 \bigr) \bigl( y_n + z_n^{1+\kappa} \bigr)
	\end{aligned}
	\end{equation}
	Combining~\eqref{eq:yn1est}, \eqref{eq:zn1est} and~\eqref{eq:un1est} we obtain with
	$\delta \coloneqq \frac{2}{N+2}$ that
	\begin{equation}\label{eq:yzn1est}
		\left\{ \begin{aligned}
			y_{n+1} & \le c_0 \; 2^{3n} \bigl( y_n^{1+\delta} + z_n^{1+\kappa} y_n^\delta \bigr) \\
			z_{n+1} & \le c_0 \; 2^{3n} \bigl( y_n + z_n^{1+\kappa} \bigr)
		\end{aligned} \right.
	\end{equation}
	for all $n \in \mathds{N}_0$.

	Next we want to estimate $y_0$ and $z_0$ for large $M$.
	On the one hand, we have
	\begin{equation}\label{eq:y0est}
		y_0 \le \frac{1}{M^2} \int_{-T}^0 \int_\Omega |u(t)|^2.
	\end{equation}
	On the other hand, similarly to~\eqref{eq:zn1est} and~\eqref{eq:un1est}, we have
	\begin{align*}
		(M - \hat{k})^2 z_0
			& \le \sum_{\ell=1}^{L_1} \biggl( \int_{-\tau_0}^0 \Bigl( \int_\Omega |u^{(\hat{k})}(t)|^{q_{1,\ell}} \Bigr)^{\frac{r_{1,\ell}}{q_{1,\ell}}} \biggr)^{\frac{2}{r_{1,\ell}}} \\
				& \qquad + \sum_{\ell=1}^{L_2} \biggl( \int_{-\tau_0}^0 \Bigl( \int_{\partial\Omega} |u^{(\hat{k})}(t)|^{q_{2,\ell}} \Bigr)^{\frac{r_{2,\ell}}{q_{2,\ell}}} \biggr)^{\frac{2}{r_{2,\ell}}} \\
			& \le c \|u^{(\hat{k})}\|_{Q(\tau_0)}^2 \\
			& \le \frac{4\gamma}{T} \int_{-T}^0 \int_\Omega |u^{(\hat{k})}(t)|^2
				+ \gamma \hat{k}^2 \Bigl( T \bigl|\Omega\bigr|^{\frac{r_1}{q_1}} \Bigr)^{\frac{2(1+\kappa)}{r_1}}
				+ \gamma \hat{k}^2 \Bigl( T \bigl|\partial\Omega\bigr|^{\frac{r_2}{q_2}} \Bigr)^{\frac{2(1+\kappa)}{r_2}},
	\end{align*}
	so that
	\begin{equation}\label{eq:z0est}
		z_0 \le \frac{c_1}{(M - \hat{k})^2} \Bigl( \int_{-T}^0 \int_\Omega |u(t)|^2 + \hat{k}^2 \Bigr)
	\end{equation}
	for all $M \ge \hat{k}$.
	Define $d \coloneqq \min\{ \delta, \frac{\kappa}{1+\kappa} \}$ and
	\[
		\lambda \coloneqq \min\bigl\{ (2c_0)^{-\frac{1}{\delta}} 2^{-\frac{3}{\delta d}}, (2c_0)^{-\frac{1+\kappa}{\kappa}} 2^{-\frac{3}{\kappa d}} \bigr\}.
	\]
	Then for
	\begin{equation}\label{eq:Mdef}
	\begin{aligned}
		M & \coloneqq \max\biggl\{ \lambda^{-\frac{1}{2}} \Bigl( \int_{-T}^0 \int_\Omega |u(t)|^2 \Bigr)^{\frac{1}{2}}, \;
			\hat{k} + \lambda^{\frac{-1}{2(1+\kappa)}} c_1^{\frac{1}{2}}\Bigl( \int_{-T}^0 \int_\Omega |u(t)|^2 + \hat{k}^2 \Bigr)^{\frac{1}{2}} \biggr\} \\
			& \le c_2 \, \Bigl( \int_{-T}^0 \int_\Omega |u(t)|^2 + \hat{k}^2 \Bigr)^{\frac{1}{2}}
	\end{aligned}
	\end{equation}
	we obtain from~\eqref{eq:y0est} and~\eqref{eq:z0est} that
	\begin{equation}\label{eq:yz0est}
		\left\{ \begin{aligned}
			y_0 & \le \lambda \\
			z_0 & \le \lambda^{\frac{1}{1+\kappa}}.
		\end{aligned} \right.
	\end{equation}

	Estimates~\eqref{eq:yzn1est} and~\eqref{eq:yz0est} show in view of Lemma~\ref{lem:verblemma} that
	$z_n \to 0$ as $n \to \infty$, which implies that
	$u(t) \le \lim_{n \to \infty} k_n = 2M$ almost everywhere on $\Omega$ for almost every
	$t \in \bigcap_{n \in \mathds{N}} [-\tau_n, 0] = [-\frac{T}{2}, 0]$ if we define $M$
	as in~\eqref{eq:Mdef}. This is~\eqref{eq:LSUest}.
\end{proof}

Theorem~\ref{thm:LSUthm1} is a local estimate in time therefore allows us to estimate
the solution of $(P_{u_0,f,g})$ independently of the initial value $u_0$.
The price is that we obtain estimates only away from $t=0$. We also need the following
modification of Theorem~\ref{thm:LSUthm1} that gives good estimates for small $t$.

\begin{corollary}\label{cor:LSUthm2}
	In the situation of Theorem~\ref{thm:LSUthm1}, assume that instead of~\eqref{eq:ukest} we even have
	\begin{align*}
		\|u^{(k)}\|_{Q(T)}^2
			& \le \gamma \int_{-T}^0 \int_\Omega |u^{(k)}(t)|^2
				+ \gamma k^2 \sum_{\ell=1}^{L_1} \Bigl( \int_{-T}^0 |A_k(t)|^{\frac{r_{1,\ell}}{q_{1,\ell}}} \Bigr)^{\frac{2(1+\kappa_{1,\ell})}{r_{1,\ell}}} \\
				& \qquad + \gamma k^2 \sum_{\ell=1}^{L_2} \Bigl( \int_{-T}^0 |B_k(t)|^{\frac{r_{2,\ell}}{q_{2,\ell}}} \Bigr)^{\frac{2(1+\kappa_{2,\ell})}{r_{2,\ell}}}
	\end{align*}
	for all $k \ge \hat{k}$.
	Then
	\[
		\esssup_{t \in [-T,0], \, x \in \Omega} u(t,x) \le c \Bigl( \int_0^T \int_\Omega |u(t)|^2 + \hat{k}^2 \Bigr)^{\frac{1}{2}}
	\]
	for all $t \in [-T,0]$, where the constant $c \ge 0$ is independent of $u$ and $\hat{k}$.
\end{corollary}
\begin{proof}
	The proof is very similar to the one of Theorem~\ref{thm:LSUthm1}. In fact, we only have to notice
	that after changing the definition of $\tau_n$ to $\tau_n \coloneqq T$ for all $n \in \mathds{N}$
	the rest of the proof carries over verbatim with the mere exception that this time
	we have $\bigcap_{n \in \mathds{N}} [-\tau_n,0] = [-T,0]$, which gives the result.
\end{proof}

Before we can check that Theorem~\ref{thm:LSUthm1} applies to the solutions of $(P_{u_0,f,g})$,
we have to supply the following tool for the calculations.

\begin{lemma}\label{lem:localderiv}
	Let $T > 0$, $u \in H^1(0,T;L^2(\Omega)) \cap L^2(0,T;H^1(\Omega))$ and $k \ge 0$. Define $u^{(k)}(t) \coloneqq (u(t) - k)^+$
	for $t \ge 0$. Then $u^{(k)} \in H^1(0,T; L^2(\Omega) ) \cap L^2(0,T; H^1(\Omega) )$ with derivative
	$(u^{(k)})_t(t) = u_t(t) \, \setone_{\lbrace u(t) > k \rbrace}$ and
	$\nabla u^{(k)}(t) = \nabla u(t) \, \setone_{\lbrace u(t) > k \rbrace}$.
	Moreover, $u^{(k)}(t)|_{\partial\Omega} = (u|_{\partial\Omega}(t) - k)^+$.
\end{lemma}
\begin{proof}
	After identifying $H^1(0,T;L^2(\Omega)) \cap L^2(0,T;H^1(\Omega))$ with $H^1( (0,T) \times \Omega )$
	up to equivalent norms in the obvious way, the formulas for the derivatives follow from
	the chain rule for weakly differentiable functions, see for example~\cite[Theorem~7.8]{GT01}.
	The assertion about the trace is true for continuous functions and thus by approximation for
	all functions under consideration.
\end{proof}

We now prove Proposition~\ref{prop:Linfbound} for classical $L^2$-solutions.
Basically, we will check that every solution of $(P_{u_0,f,g})$ satisfies~\eqref{eq:ukest}.

\begin{lemma}\label{lem:Linfbound}
	Proposition~\ref{prop:Linfbound} holds if in addition we assume that $u$ is a classical $L^2$-solution
	and $T \le T_0$, where $T_0 > 0$ depends
	only on $N$, $\Omega$, $r_1$, $q_1$, $r_2$, $q_2$ and the coefficients of the equation.
\end{lemma}
\begin{proof}
	After a linear substitution in the time variable we may consider problem $(P_{u_0,f,g})$ on $[-T,0]$
	instead of $[0,T]$, the initial value now being $u_0 = u(-T)$.
	We check the conditions of Theorem~\ref{thm:LSUthm1} with
	\begin{equation}\label{eq:khat}
		\hat{k}^2 \coloneqq \|f\|_{L^{r_1}(-T,0;L^{q_1}(\Omega))}^2 + \|g\|_{L^{r_2}(-T,0;L^{q_2}(\Omega))}^2.
	\end{equation}

	Fix $0 < \tau \le T$ and let $\zeta$ be a function in $H^1(-\tau,0)$ satisfying $0 \le \zeta(t) \le 1$ for all $t \in [-\tau,0]$.
	Assume either that $\zeta(-\tau) = 0$ or that $\tau = T$ and $u^{(k)}(-T) = 0$.
	Then for $t \in [-\tau,0]$ we have
	\begin{equation}\label{eq:deGiest}
		\begin{aligned}
			\zeta(t)^2 \cdot \frac{1}{2} \int_\Omega |u^{(k)}(t)|^2
				& = \int_{-\tau}^t \frac{\dx}{\dx s} \Bigl( \zeta(s)^2 \cdot \frac{1}{2} \int_\Omega |u^{(k)}(s)|^2 \Bigr) \\
				& = \int_{-\tau}^t \zeta(s) \zeta'(s) \int_\Omega |u^{(k)}(s)|^2 + \int_{-\tau}^t \zeta(s)^2 \int_\Omega u^{(k)}_t(s) \; u^{(k)}(s).
		\end{aligned}
	\end{equation}
	From Lemma~\ref{lem:localderiv} and the fact that $u$ is a classical $L^2$-solution of $(P_{u_0,f,g})$
	we obtain that for all $s \in [-\tau,0]$ we have
	\begin{equation}\label{eq:Auterm}
		\begin{aligned}
			\int_\Omega u^{(k)}_t(s) \; u^{(k)}(s)
				& = \int_\Omega u_t(s) \; u^{(k)}(s)
				= \int_\Omega \bigl( Au(s) + f(s) \bigr) \; u^{(k)}(s) \\
				& = \int_\Omega f(s) \, u^{(k)}(s) + \int_{\partial\Omega} g(s) \, u^{(k)}(s) - a_\beta(u(s), u^{(k)}(s)).
		\end{aligned}
	\end{equation}

	We now estimate the right hand side of~\eqref{eq:Auterm}.
	From Lemma~\ref{lem:localderiv}, \eqref{eq:aelliptic} and Young's inequality we obtain that
	\begin{align*}
		& a_\beta(u(s), u^{(k)}(s)) \\
			& \quad = a_\beta(u^{(k)}(s), u^{(k)}(s)) + \sum_{j=1}^N \int_\Omega b_j \; k D_ju^{(k)}(s)
				+ \int_\Omega d \, k u^{(k)}(s) + \int_{\partial\Omega} \beta \, k u^{(k)}(s) \\
			& \quad \ge \frac{\mu}{2} \int_\Omega |\nabla u^{(k)}(s)|^2 - \omega \int_\Omega |u^{(k)}(s)|^2
				- \frac{k^2}{\mu} \sum_{j=1}^N \int_{A_k(s)} |b_j|^2 - \frac{\mu}{4} \int_\Omega |\nabla u^{(k)}(s)|^2 \\
				& \qquad -  \int_{A_k(s)} |d| \; \bigl( |u^{(k)}(s)|^2 + k^2 \bigr)
				- \int_{B_k(s)} |\beta| \; \bigl( |u^{(k)}(s)|^2 + k^2 \bigr).
	\end{align*}
	Using~\eqref{eq:Auterm} and again Young's inequality this gives
	\begin{align}
		\int_\Omega u^{(k)}_t(s) \; u^{(k)}(s)
			& \le -\frac{\mu}{4} \int_\Omega |\nabla u^{(k)}(s)|^2
				+ \int_{A_k(s)} \bigl( \tfrac{1}{k} |f(s)| + \mathcal{D}_0 \bigr) \bigl( |u^{(k)}(s)|^2 + k^2 \bigr) \nonumber \\
				& \qquad + \int_{B_k(s)} \bigl( \tfrac{1}{k} |g(s)| + |\beta| \bigr) \bigl( |u^{(k)}(s)|^2 + k^2 \bigr) \label{eq:utuest}
	\end{align}
	with
	\[
		\mathcal{D}_0 \coloneqq \omega + \frac{1}{\mu} \sum_{j=1}^N |b_j|^2 + |d| \in L^{\frac{q}{2}}(\Omega),
	\]
	where $q > N$.
	Plugging~\eqref{eq:utuest} into~\eqref{eq:deGiest} and varying over $t$ we arrive at the estimate
	\begin{equation}\label{eq:fullest1}
		\begin{aligned}
			& \min\{ \tfrac{1}{2}, \tfrac{\mu}{4} \} \| \zeta u^{(k)}\|_{Q(\tau)}^2 \\
				& \qquad \le \sup_{-\tau \le t \le 0} \Bigl( \zeta(t)^2 \cdot \frac{1}{2} \int_\Omega |u^{(k)}(t)|^2 \Bigr) + \frac{\mu}{4} \int_{-\tau}^0 \zeta(s)^2 \int_\Omega |\nabla u^{(k)}|^2 \\
				& \qquad \le \|\zeta'\|_{L^\infty(-\tau,0)} \int_{-\tau}^0 \int_\Omega |u^{(k)}(s)|^2 \\
					& \qquad\qquad + \int_{-\tau}^0 \int_{A_k(s)} \bigl( \tfrac{1}{k} |f(s)| + \mathcal{D}_0 \bigr) \cdot \bigl( \zeta(s)^2 |u^{(k)}(s)|^2 + k^2 \bigr)  \\
					& \qquad\qquad + \int_{-\tau}^0 \int_{B_k(s)} \bigl( \tfrac{1}{k} |g(s)| + |\beta| \bigr) \cdot \bigl( \zeta(s)^2 |u^{(k)}(s)|^2 + k^2 \bigr) 
		\end{aligned}
	\end{equation}

	We estimate the right hand side of~\eqref{eq:fullest1}. Define $\kappa_1 > 0$ and $\kappa_2 > 0$ by
	\begin{equation}\label{eq:qrrelk}
		\frac{1}{r_1} + \frac{N}{2q_1} = 1 - \frac{\kappa_1 N}{2}
		\qquad\text{and}\qquad
		\frac{1}{r_2} + \frac{N-1}{2q_2} = \frac{1}{2} - \frac{\kappa_2 N}{2}.
	\end{equation}
	With $\bar{r}_1 \coloneqq \frac{2r_1}{r_1 - 1}$ and $\bar{q}_1 \coloneqq \frac{2q_1}{q_1 - 1}$ we obtain from
	H\"older's inequality that
	\begin{align*}
		& \int_{-\tau}^0 \int_{A_k(s)} |f(s)| \cdot \zeta(s)^2 |u^{(k)}(s)|^2 \\
			& \qquad \le \|f\|_{L^{r_1}(-\tau,0;L^{q_1}(\Omega))}
					\| \zeta u^{(k)} \|_{L^{\bar{r}_1}(-\tau,0;L^{\bar{q}_1}(\Omega))}^2 \\
			& \qquad \le \hat{k} \| \zeta u^{(k)} \|_{L^{(1+\kappa_1)\bar{r}_1}(-\tau,0;L^{(1+\kappa)\bar{q}_1}(\Omega))}^2
				\|\setone_{A_k}\|_{L^{\frac{\kappa_1+1}{\kappa_1}\bar{r}_1}(-\tau,0; L^{\frac{\kappa_1+1}{\kappa_1}\bar{q}_1}(\Omega))}^2
	\end{align*}
	The last factor tends to zero as $\tau \to 0$.
	Since moreover $\frac{1}{(1+\kappa_1)\bar{r}_1} + \frac{N}{2(1+\kappa_1)\bar{q}_1} = \frac{N}{4}$ by~\eqref{eq:qrrelk},
	we deduce from~\eqref{eq:aniso} that
	\[
		\int_{-\tau}^0 \int_{A_k(s)} |f(s)| \cdot \zeta(s)^2 |u^{(k)}(s)|^2
			\le \frac{\hat{k}}{8} \min\{ \tfrac{1}{2}, \tfrac{\mu}{4} \} \|\zeta u^{(k)}\|_{Q(\tau)}^2
	\]
	if $\tau$ is sufficiently small, say $\tau \le T_0$, where $T_0$ depends on $\mu$, $N$, $\Omega$, $\kappa_1$, $r_1$, $q_1$.
	Similarly, since $\frac{2q}{q-2} < \frac{2N}{N-2}$ we obtain that
	\begin{align*}
		\int_{-\tau}^0 \int_{A_k(s)} \mathcal{D}_0 \cdot \zeta(s)^2 |u^{(k)}(s)|^2
			& \le \|\mathcal{D}_0\|_{L^{\frac{q}{2}}(\Omega)} \|\zeta u^{(k)}\|_{L^2(-\tau,0;L^{\frac{2q}{q-2}}(\Omega))}^2 \\
			& \le \frac{1}{8} \min\{ \tfrac{1}{2}, \tfrac{\mu}{4} \} \|\zeta u^{(k)}\|_{Q(\tau)}^2
	\end{align*}
	for $\tau \le T_0$ with some possibly smaller $T_0 > 0$ that depends in addition on $\mathcal{D}_0$ and $q$.

	Analogously, with $\bar{r}_2 \coloneqq \frac{2r_2}{r_2-1}$ and $\bar{q}_2 \coloneqq \frac{2q_2}{q_2-1}$ we have
	\begin{align*}
		& \int_{-\tau}^0 \int_{B_k(s)} |g(s)| \cdot \zeta(s)^2 |u^{(k)}(s)|^2 \\
			& \qquad \le \hat{k} \| \zeta u^{(k)} \|_{L^{(1+\kappa_2)\bar{r}_2}(-\tau,0;L^{(1+\kappa_2)\bar{q}_2}(\partial\Omega))}^2
				\|\setone_{B_k}\|_{L^{\frac{\kappa_1+1}{\kappa_1}\bar{r}_2}(-\tau,0; L^{\frac{\kappa_1+1}{\kappa_1}\bar{q}_2}(\partial\Omega))}^2 \\
			& \qquad \le \frac{\hat{k}}{8} \min\{ \tfrac{1}{2}, \tfrac{\mu}{4} \} \|\zeta u^{(k)}\|_{Q(\tau)}^2
	\end{align*}
	and since $\frac{2(q-1)}{q-2} < \frac{2(N-1)}{N-2}$ also
	\begin{align*}
		\int_{-\tau}^0 \int_{B_k(s)} |\beta| \cdot \zeta(s)^2 |u^{(k)}(s)|^2
			& \le \|\beta\|_{L^{q-1}(\partial\Omega)} \|\zeta u^{(k)}\|_{L^2(-\tau,0; L^{\frac{2(q-1)}{q-2}}(\partial\Omega))}^2 \\
			& \le \frac{1}{8} \min\{ \tfrac{1}{2}, \tfrac{\mu}{4} \} \|\zeta u^{(k)}\|_{Q(\tau)}^2
	\end{align*}
	for $\tau \le T_0$, where this new $T_0$ depends also on $r_2$, $q_2$, $\kappa_2$ and $\beta$.

	Combining the latter estimates with~\eqref{eq:fullest1} we obtain that
	\begin{align}
		\|\zeta u^{(k)}\|_{Q(\tau)}^2
			& \le c_\mu \|\zeta'\|_{L^\infty(-\tau,0)} \int_{-\tau}^0 \int_\Omega |u^{(k)}(s)|^2
				+ c_\mu k^2 \int_{-\tau}^0 \int_{A_k(s)} \bigl( \tfrac{1}{k} |f(s)| + \mathcal{D}_0 \bigr) \nonumber \\
				& \qquad + c_\mu k^2 \int_{-\tau}^0 \int_{B_k(s)} \bigl( \tfrac{1}{k} |g(s)| + |\beta| \bigr)  \label{eq:fullest2}
	\end{align}
	if $\tau \le T_0$ and $k \ge \hat{k}$, where $c_\mu$ depends only on $\mu$.

	Now we estimate for $k \ge \hat{k}$
	\begin{align*}
		\int_{-\tau}^0 \int_{A_k(s)} \tfrac{1}{k} |f(s)|
			& \le \tfrac{1}{k} \|f\|_{L^{r_1}(-\tau,0; L^{q_1}(\Omega))} \|\setone_{A_k}\|_{L^{\frac{r_1}{r_1-1}}(-\tau,0; L^{\frac{q_1}{q_1-1}}(\Omega))} \\
			& \le \|\setone_{A_k}\|_{L^{\frac{r_1}{r_1-1}}(-\tau,0; L^{\frac{q_1}{q_1-1}}(\Omega))} \\
			& = \|\setone_{A_k}\|_{L^{r_{1,1}}(-\tau,0; L^{q_{1,1}}(\Omega))}^{2(1+\kappa_{1,1})},
	\end{align*}
	with $\kappa_{1,1} \coloneqq \kappa_1$, $r_{1,1} \coloneqq 2(1+\kappa_1) \frac{r_1}{r_1-1}$
	and $q_{1,1} \coloneqq 2(1+\kappa_1) \frac{q_1}{q_1-1}$
	and similarly
	\begin{align*}
		\int_{-\tau}^0 \int_{A_k(s)} \mathcal{D}_0
			& \le \|\mathcal{D}_0\|_{L^{\frac{q}{2}}(\Omega)} \; \|\setone_{A_k}\|_{L^1(-\tau,0; L^{\frac{q}{q-2}}(\Omega))} \\
			& = \|\mathcal{D}_0\|_{L^{\frac{q}{2}}(\Omega)} \; \|\setone_{A_k}\|_{L^{r_{1,2}}(-\tau,0; L^{q_{1,2}}(\Omega))}^{2(1+\kappa_{1,2})}
	\end{align*}
	with $\kappa_{1,2} \coloneqq \frac{2(q-N) + (q-2)N}{qN}$,
	$r_{1,2} \coloneqq 2(1+\kappa_{1,2})$ and $q_{1,2} \coloneqq 2(1+\kappa_{1,2}) \frac{q}{q-2}$.
	Analogously,
	\[
		\int_{-\tau}^0 \int_{B_k(s)} \tfrac{1}{k} |g(s)|
			\le \|\setone_{B_k}\|_{L^{r_{2,1}}(-\tau,0;L^{q_2}(\partial\Omega))}^{2(1+\kappa_{2,1})}
	\]
	with $\kappa_{2,1} \coloneqq \kappa_2$, $r_{2,1} \coloneqq 2(1+\kappa_{2,1})\frac{r_2}{r_2-1}$
	and $q_{2,1} \coloneqq 2(1+\kappa_{2,1}) \frac{q_2}{q_2-1}$
	and
	\[
		\int_{-\tau}^0 \int_{B_k(s)} |\beta|
			\le \|\beta\|_{L^{q-1}(\partial\Omega)} \|\setone_{B_k}\|_{L^{r_{2,2}}(-\tau,0;L^{q_{2,2}}(\partial\Omega))}^{2(1+\kappa_{2,2})}
	\]
	with $\kappa_{2,2} \coloneqq \frac{N(q-N) + 2(N-1)}{(q-1)N}$, $r_{2,2} \coloneqq 2(1+\kappa_{2,2})$ and $q_{2,2} \coloneqq 2(1+\kappa_{2,2}) \frac{q-1}{q-2}$.
	Thus~\eqref{eq:fullest2} yields
	\begin{align}
		\|\zeta u^{(k)}\|_{Q(\tau)}^2
			& \le c_\mu \|\zeta'\|_{L^\infty(-\tau,0)} \int_{-\tau}^0 \int_\Omega |u^{(k)}(s)|^2
				+ c k^2 \sum_{\ell=1}^2 \Bigl( \int_{-\tau}^0 |A_k(s)|^{\frac{r_{1,\ell}}{q_{1,\ell}}} \Bigr)^{\frac{2(1+\kappa_{1,\ell})}{r_{1,\ell}}} \nonumber \\
				& \qquad + c k^2 \sum_{\ell=1}^2 \Bigl( \int_{-\tau}^0 |B_k(s)|^{\frac{r_{2,\ell}}{q_{2,\ell}}} \Bigr)^{\frac{2(1+\kappa_{2,\ell})}{r_{2,\ell}}} \label{eq:fullest3}
	\end{align}
	Moreover, \eqref{eq:qrrelk} implies that the parameters $r_{i,\ell}$ and $q_{i,\ell}$
	satisfy~\eqref{eq:qrcond} for $i=1,2$ and $\ell=1,2$ as elementary calculations show.

	If we pick $\zeta(t) \coloneqq \frac{t + \tau}{\sigma \tau}$
	for $t \in [-\tau,-(1-\sigma)\tau]$ and $\zeta(t) \coloneqq 1$ for $t \in [-(1-\sigma)\tau, 0]$ with some
	given $\sigma \in (0,\frac{1}{2})$, we have
	\[
		\|u^{(k)}\|_{Q((1-\sigma)\tau)}^2 \le \|\zeta u^{(k)}\|_{Q(\tau)}^2
	\]
	and $\|\zeta'\|_{L^\infty(-\tau,0)} \le \frac{1}{\sigma \tau}$
	if $T \le T_0$, where $c$ depends only on $\mu$, $\mathcal{D}_0$ and $\beta$.
	Thus~\eqref{eq:fullest3} implies~\eqref{eq:ukest}. Hence by Theorem~\ref{thm:LSUthm1} applied
	to $u$ and $-u$, the latter being a classical solution of $(P_{-u_0,-f,-g})$, we obtain~\eqref{eq:Linfbound}.

	If in addition $u(-T) = 0$, then we can set $\tau \coloneqq T$ and choose
	$\zeta(t) \coloneqq 1$ for all $t \in [-T,0]$.
	Now using Corollary~\ref{cor:LSUthm2} instead of Theorem~\ref{thm:LSUthm1}, we obtain~\eqref{eq:Linfboundglob}
	from~\eqref{eq:fullest3} like above.
\end{proof}

We finally make the step from classical $L^2$-solutions to weak solutions and drop the assumption
that $T$ be small enough, thus proving Proposition~\ref{prop:Linfbound}.

\begin{proof}[Proof of Proposition~\ref{prop:Linfbound}]
	Let $u$ be the weak solution of $(P_{u_0,f,g})$.
	Pick a sequence $(u_{0,n})$ in $D(A_{2,h}^2)$ that satisfies $u_{0,n} \to u_0$
	in $L^2(\Omega)$, which exists since by Proposition~\ref{prop:semigroup} the operator $A_{2,h}$
	is a generator of a strongly continuous semigroup and hence densely defined.
	Pick sequences $(f_n)$ and $(g_n)$ in $\mathrm{C}^2( [0,T]; L^\infty(\Omega) )$ and $\mathrm{C}^2( [0,T]; L^\infty(\partial\Omega) )$, respectively,
	that satisfy $f_n \to f$ in $L^{r_1}(0,T; L^{q_1}(\Omega))$ and
	$g_n \to g$ in $L^{r_2}(0,T; L^{q_2}(\partial\Omega))$, while $f_n(0) = 0$ and $g_n(0) = 0$
	for all $n \in \mathds{N}$.
	Then problem $(P_{u_{0,n},f_n,g_n})$ has a unique classical $L^2$-solution $u_n$ by Proposition~\ref{prop:respossol},
	and as in the proof of Theorem~\ref{thm:mainex} we see that $u_n \to u$ in $\mathrm{C}([0,T]; L^2(\Omega)) \cap L^2(0,T; H^1(\Omega))$.

	Pick $T_0 > 0$ as in Lemma~\ref{lem:Linfbound}. Shrinking $T_0$, if necessary, we can assume that $T_0 \le T$.
	Let $I \subset [\frac{T_0}{2}, T_0]$ be an interval of length at most $\frac{T_0}{2}$.
	Applying~\eqref{eq:Linfbound} for the classical $L^2$-solutions $u_n$ and $u_n - u_m$ on $I$,
	which is allowed by Lemma~\ref{lem:Linfbound}, we obtain that
	\begin{equation}\label{eq:unestinf}
		\|u_n\|_{L^\infty(I; L^\infty(\Omega))}^2
			\le c \int_0^T \int_\Omega |u_n(s)|^2
				+ c \|f_n\|_{L^{r_1}(0,T;L^{q_1}(\Omega))}^2 + c \|g_n\|_{L^{r_2}(0,T;L^{q_2}(\Omega))}^2
	\end{equation}
	and that $(u_n|_I)$ is a Cauchy sequence in $L^\infty(I; L^\infty(\Omega))$.
	Hence $u_n \to u$ in $L^\infty(I; L^\infty(\Omega))$ and passing to the limit in~\eqref{eq:unestinf}
	we have
	\begin{equation}\label{eq:uestinf}
		\|u\|_{L^\infty(I; L^\infty(\Omega))}^2
			\le c \int_0^T \int_\Omega |u(s)|^2
				+ c \|f\|_{L^{r_1}(0,T;L^{q_1}(\Omega))}^2 + c \|g\|_{L^{r_2}(0,T;L^{q_2}(\Omega))}^2
	\end{equation}
	Covering $[\frac{T}{2}, T]$ by finitely many intervals of length at most $\frac{T_0}{2}$
	and using~\eqref{eq:uestinf} for each of these intervals we obtain~\eqref{eq:Linfbound}.

	If in addition $u_0 = 0$, then we can pick $u_{0,n} \coloneqq 0$
	and the same strategy as above yields that
	\[
		\|u\|_{L^\infty(0,T_0; L^\infty(\Omega))}^2
			\le c \int_0^T \int_\Omega |u(s)|^2
				+ c \|f\|_{L^{r_1}(0,T;L^{q_1}(\Omega))}^2 + c \|g\|_{L^{r_2}(0,T;L^{q_2}(\Omega))}^2.
	\]
	Using in addition~\eqref{eq:Linfbound} to estimate $\|u\|_{L^\infty(I; L^\infty(\Omega))}$ for finitely many
	intervals $I$ of length $\frac{T_0}{2}$ that cover $[T_0,T]$, we have proved also~\eqref{eq:Linfboundglob}.
\end{proof}

\section*{Acknowledgments}

The author is grateful to Wolfgang Arendt for many fruitful discussions.

\bibliographystyle{amsplain}
\bibliography{inhomneumann}

\end{document}